\documentclass{amsart}
\usepackage{amsmath, amssymb, graphicx, epsfig, verbatim, psfrag, color,amscd,stmaryrd,rsfso} 
\usepackage{tikz}
\usepackage{amsthm}
\usepackage{enumitem}
\usetikzlibrary{arrows}
\usepackage[T1]{fontenc}
\usepackage[utf8]{inputenc}

\usepackage{mathrsfs}

\newcommand\OurRing{\mathcal R}
\newcommand\mfacZ{\gamma^1_0}
\newcommand\mfac{\gamma^1}
\newcommand\Hup{\mathcal{H}^{\wedge}}

\newcommand\state{\kappa}

\usetikzlibrary{matrix,arrows,positioning}
\tikzset{cdlabel/.style={above,sloped,
    execute at begin node=$\scriptstyle,execute at end node=$}}
\tikzset{algarrow/.style={->, thick}}
\tikzset{blgarrow/.style={->, thick}}
\tikzset{clgarrow/.style={->, thick}}
\tikzset{tensoralgarrow/.style={double, double equal sign distance, -implies}}
\tikzset{tensorblgarrow/.style={double, double equal sign distance, -implies}}
\tikzset{tensorclgarrow/.style={double, double equal sign distance, -implies}}
\tikzset{modarrow/.style={->, dashed}}
\tikzset{othmodarrow/.style={->, thick}}
\tikzset{Amodar/.style={->, dashed}}
\tikzset{Dmodar/.style={->, dashed}}

\newcommand\HD{\mathcal H}

\newcommand\CFAa{\widehat{\mathrm{CFA}}}
\newcommand\CFDa{\widehat{\mathrm{CFD}}}

\newcommand\Min{\GenMin}
\newcommand\GenMin{\mho}
\newcommand\KHm{H^-}
\newcommand\KHwz{\KH}

\newcommand\KHa{\widehat H}

\newcommand\KC{{\mathcal C}}
\newcommand\KH{{H}}

\newcommand\Diag{\mathcal D}
\newcommand\States{\mathfrak{S}}
\newcommand\HFKm{\mathrm{HFK}^-}

\def\endproof{\relax\ifmmode\expandafter\endproofmath\else
  \unskip\nobreak\hfil\penalty50\hskip.75em\hbox{}\nobreak\hfil\bull
  {\parfillskip=0pt \finalhyphendemerits=0 \bigbreak}\fi}
\def\endproofmath$${\eqno\bull$$\bigbreak}
\def\bull{\vbox{\hrule\hbox{\vrule\kern3pt\vbox{\kern6pt}\kern3pt\vrule}\hrule}}

\newcommand\CanonDD{\mathcal K}
\newcommand\North{\mathbf N}
\newcommand\South{\mathbf S}
\newcommand\East{\mathbf E}
\newcommand\West{\mathbf W}

\newcommand\Pos{\mathcal P}

\newcommand\HFKa{\widehat{\mathrm{HFK}}}

\newtheorem{thm}{Theorem}[section]

\newtheorem{cor}[thm]{Corollary}

\newtheorem{prop}[thm]{Proposition}
\newtheorem{defn}[thm]{Definition}

\newtheorem{example}[thm]{Example}

\numberwithin{equation}{section}

\newcommand\OneHalf{\frac{1}{2}}

\newcommand\IdempRing{\mathbf{I}}
\newcommand\Idemp[1]{\mathbf{I}_{#1}}
\newcommand\DT{\boxtimes}
\newcommand\x{\mathbf x}

\newcommand\y{\mathbf y}

\newcommand\lsup[2]{^{#1}{#2}}

\newcommand{\Ideal}{\mathcal I}

\newcommand\z{\mathbf z}

 \newcommand{\Z}{\mathbb Z}  \newcommand{\Q}{\mathbb Q} \newcommand{\R}{\mathbb R}

\newcommand{\HFa}{{\widehat {\rm {HF}}}}
\newcommand{\HFm}{{\rm {HF}^-}}

\newcommand{\CFa}{{\widehat {\rm {CF}}}}
\newcommand{\CFm}{{\rm {CF}} ^-}

\newcommand\XX{\mathbf X}
\newcommand\YY{\mathbf Y}
\newcommand\ZZ{\mathbf Z}
\newcommand\Max{\Omega}

\newcommand\Field{\mathbb F}
\DeclareMathOperator{\Hom}{Hom}

\DeclareMathOperator{\Id}{Id}

\newcommand\Alg{\mathcal A}
\newcommand\Blg{\mathcal B}

\newcommand\Ainf{{\mathcal A}_{\infty}}
\newcommand\Ainfty\Ainf
\newcommand\Zmod[1]{{\mathbb Z}/{#1}{\mathbb Z}}

\newcommand\gr{\mathbf{gr}}

\newcommand\Partition{P}
\newcommand\DuAlg{{\mathcal A}'}

\newcommand\Matching{M}

\makeatletter
\renewenvironment{proof}[1][\proofname]{\par
\pushQED{\qed}%
\normalfont \topsep6\p@\@plus6\p@\relax
\trivlist
\item\relax
{\bf#1\@addpunct{.}}\hspace\labelsep\ignorespaces
}{%
\popQED\endtrivlist\@endpefalse
}
\makeatother

\begin{document}
%\title{Holomorphic disks, algebra, and knot invariants}
\title{An overview of knot Floer homology}

\author[Peter S. Ozsv\'ath]{Peter Ozsv\'ath}
\thanks {PSO was supported by NSF grant number DMS-1405114}
\address {Department of Mathematics, Princeton University\\ Princeton, New Jersey 08544} 
\email {petero@math.princeton.edu}

\author[Zolt{\'a}n Szab{\'o}]{Zolt{\'a}n Szab{\'o}}
\thanks{ZSz was supported by NSF grant number DMS-1606571}
\address{Department of Mathematics, Princeton University\\ Princeton, New Jersey 08544}
\email {szabo@math.princeton.edu}

%\begin{abstract}
%  Knot Floer homology is a knot invariant defined using holomorphic
%  curves.  In recent work, the authors defined another knot invariant,
%  called ``bordered knot Floer homology''.  The present paper
%  identifies these two knot invariants.
%\end{abstract}
\maketitle

\newcommand\Cobar{\mathrm{Cobar}}
\newcommand\Interior[1]{#1^\circ}
\newcommand\gen[1]{\gamma_{#1}}
\newcommand\NN{\mathbf N}
\newcommand\WW{\mathbf W}
\newcommand\EE{\mathbf E}
\renewcommand\SS{\mathbf S}

\newcommand\dbar{\overline\partial}
\newcommand\gammas{\boldsymbol{\gamma}}
\newcommand\alphas{\boldsymbol{\alpha}}
\newcommand\betas{\boldsymbol{\beta}}
\newcommand\Ta{{\mathbb T}_{\alpha}}
\newcommand\Tb{{\mathbb T}_{\beta}}
\newcommand\Sym{\mathrm{Sym}}
\newcommand\CDisk{\mathbb D}
\newcommand\dm{\partial^-}
\newcommand\da{\widehat\partial}
\newcommand\Mas{\mu}
\newcommand\ModFlow{\mathcal{M}}
\newcommand\CFKa{\widehat{\mathrm{CFK}}}
\newcommand\CFKwz{\mathrm{CFK}}
\newcommand\HFKwz{\mathrm{HFK}}
\newcommand\CFKm{{\mathrm{CFK}}^-}
\newcommand\uCFK{{\mathrm{CFK}}'}
\newcommand\tHFK{\mathrm{tHFK}}
\newcommand\tCFK{\mathrm{tCFK}}
\newcommand\uHFK{{\mathrm{HFK}}'}
\newcommand\Unknot{\mathcal{U}}
\newcommand\orL{\vec{L}}
\newcommand\uda{\partial_K'}
\newcommand\dt{\partial_K^t}
\newcommand\mirror{\mathfrak{m}}
\newcommand\gTopSlice{g^{\mathrm{top}}_4}
\newcommand\Conc{\mathcal{C}}
\newcommand\ConcTS{\mathcal{C}_{TS}}
\newcommand\ConcTop{\mathcal{C}^{\mathrm{top}}}
\newcommand\ws{\mathbf{w}}
\newcommand\zs{\mathbf{z}}
% \section{Introduction}

Knot Floer homology is an invariant for knots discovered by the
authors~\cite{Knots} and, independently, Jacob
Rasmussen~\cite{RasmussenThesis}.  The discovery of this invariant
grew naturally out of studying how a certain three-manifold invariant,
{\em Heegaard Floer homology}~\cite{HolDisk}, changes as the three-manifold undergoes
Dehn surgery along a knot. Since its original definition, thanks  to
the contributions of many researchers, knot Floer homology 
has emerged as a useful tool for studying knots in its own right.
We give here a few selected highlights of this theory,
and then move on to some new algebraic developments in the
computation of knot Floer homology.

\section{Motivation for the construction}
Since the work of Donaldson, gauge theory has emerged as the central
tool for understanding differential topology in dimension
four. Donaldson's pioneering work from the 1980's used the moduli
space of solutions to the anti-self-dual Yang-Mills equations -- or
``instantons'' -- to construct diffeomorphism
invariants of four-dimensional manifolds~\cite{Donaldson,DonaldsonPolynomials}.  Donaldson used these invariants to
discover completely unexpected phenomena in four-dimensional topology,
including a deep connection between the smooth topology of algebraic
surfaces and their algebraic geometry, leading to a number of
breakthroughs in the
field~\cite{Chambers,Kotschick,FriedmanMorgan,FintushelSternDefinite,GompfMrowka,Szabo,DonaldsonKronheimer,KMPolyStruct}.

A corresponding invariant for three-dimensional manifolds was
introduced by Andreas Floer, {\em instanton Floer homology}, which is
the homology group of a chain complex whose generators are $SU(2)$
representations of the fundamental group of the three-manifold $Y$
(modulo conjugation), and whose differential counts instantons on
$\R\times Y$; see~\cite{Floer,DonaldsonFloer}. 
Floer homology can be used as a tool for computing
Donaldson's invariants~\cite{FintSternK3}.

Floer formulated his instanton homology theory as a kind of
infinite-dimensional Morse theory, akin to his earlier {\em Lagrangian
  Floer homology}, which is an invariant for a symplectic manifold,
equipped with a pair of Lagrangian submanifolds~\cite{FloerLagrangian}; see also~\cite{FOOO}. 
In~\cite{AtiyahFloer}, Atiyah proposed a relationship between these
two invariants, which is now known as the ``Atiyah-Floer
conjecture''. The starting point of this conjecture is a
three-manifold equipped with a Heegaard splitting.  The ``character
variety'' of the Heegaard surface $\Sigma$, which is the space of
representations of $\pi_1(\Sigma)$ into $SU(2)$ modulo conjugation, is
equipped with a pair of Lagrangian subspaces, the spaces of
representations that extend over each handlebody. The Atiyah-Floer
conjecture states that the Lagrangian Floer homology of these
character varieties should agree with the instanton homology of the
underlying three-manifold $Y$; compare~\cite{TaubesCasson}. This statement is still a little vague:
$SU(2)$ instanton homology is defined for three-manifolds with
$H_1(Y;\Z)=0$, and the spaces involved on the symplectic side are
singular. Nonetheless, the conjectured relationship has spurred a
great deal of mathematical activity~\cite{DostSal,Wehrheim}.

In 1994, four-manifold topology was revolutionized by the introduction
of the Seiberg-Witten equations, a new partial differential equation
coming from physics~\cite{Witten}. The moduli spaces of solutions to
these equations could be used to construct invariants of smooth
four-manifolds, just as the anti-self-dual equations are used in
Donaldson's theory. Many theorems  proved earlier using
Donaldson's invariants had easier proofs and generalizations using the
newly introduced Seiberg-Witten invariants~\cite{DonaldsonSW}.  The Seiberg-Witten
invariants also elucidated the relationship between the differential
topology of symplectic manifolds and their symplectic properties,
resulting in Clifford Taubes' celebrated proof that identified the
Gromov-Witten invariants of a symplectic manifold with their
Seiberg-Witten invariants~\cite{TaubesSympI,TaubesSympII,TauSWGromov}.

Considerable work went into formulating a three-dimensional analogue
of the Seiberg-Witten invariants. A definitive construction was given
by Peter Kronheimer and Tomasz Mrowka in their monograph~\cite{KMbook};
see also~\cite{MarcolliWang,Froyshov,Manolescu}.

Heegaard Floer homology~\cite{HolDisk} grew out of our attempts to
concretely understand the geometric underpinnings of Seiberg-Witten
theory.  A motivating problem was to find the analogue of the
``Atiyah-Floer conjecture'': what Lagrangian Floer construction could
possibly recapture the Seiberg-Witten invariants for three-manifolds?
A clue was offered by the the following observation: the space of
stationary solutions to the (suitably perturbed) Seiberg-Witten
equations on $\R\times\Sigma$ is identified the moduli space of
``vortices'' on $\Sigma$ with some charge $d$, which in turn, by early
work of Taubes, is identified with the $d$-fold symmetric product of
$\Sigma$, the space of unordered $d$-tuples of points in
$\Sigma$, denoted $\Sym^d(\Sigma)$.

It was proved in~\cite{HolDisk} that Heegaard Floer homology is a
well-defined three-manifold invariant, enjoying many of the properties
of Seiberg-Witten theory. Although it was designed to be isomorphic to
construction in Seiberg-Witten theory, the conjectural equivalence of
these two theories was verified many years after their formulation, in
the work of Cagatay Kutluhan, Yi-Jen Lee, Taubes~\cite{KLTI} and
Vincent Colin, Paolo Ghiggini, and Ko Honda~\cite{CGH}.

{\bf Acknowledgements.}  We would like to thank Andr{\'a}s Stipsicz
for his suggestions on an early draft of this paper.  The work of
Simon Donaldson has had a great impact on our research. Both of our PhD
theses were based on computing Donaldson's invariant for
four-manifolds; and indeed his theory has served as an inspiration to
us ever since!

\section{Statement of the symplectic constructions}

We sketch now the construction of Heegaard Floer homology, and its
corresponding knot invariant. Before doing this, we recall some
topological preliminaries.

Let $\Sigma$ be a surface of genus $g$. A {\em complete set of
  attaching circles for $\Sigma$} is a $g$-tuple of pairwise disjoint,
  homologically linearly independent simple, closed curves.  A
  complete set of attaching circles specifies a handlebody
  $U_{\gammas}$ whose boundary is identified with $\Sigma$, so that  the
  attaching circles bound disjoint, embedded disks in 
  $U_{\gammas}$.

A {\em Heegaard splitting} of a connected, closed, oriented three-manifold $Y$ is
a decomposition of $Y$ as the union of two handlebodies, glued along
their boundary.  Combinatorially, a Heegaard splitting is specified
by a {\em Heegaard diagram}, which consists of a triple $(\Sigma,\alphas,\betas)$, where $\Sigma$ is an oriented surface,
$\alphas=\{\alpha_1,\dots,\alpha_g\}$ and
$\betas=\{\beta_1,\dots,\beta_g\}$ are two complete sets of attaching
circles for $\Sigma$.  Heegaard diagrams can be thought of from the
perspective of Morse theory~\cite{MorseTheory,Milnor}, as
follows. If $Y$ is equipped with a self-indexing Morse function $f$
and a gradient-like vector field $v$, we can let $\Sigma$ be
$f^{-1}(3/2)$, and $\alphas$ is the locus of points in $\Sigma$ that
flow out of the index one critical points under $v$, and $\betas$ is
the locus of points in $\Sigma$ that flow into the the index two
critical points.  We will typically work with {\em pointed Heegaard
diagrams}, which consist of data $\HD=(\Sigma,\alphas,\betas,w)$, where
$(\Sigma,\alphas,\betas)$ is a Heegaard diagram, and $w\in \Sigma$ is
an auxiliary basepoint in $\Sigma$ that is disjoint from all the
$\alpha_i$ and the $\beta_j$.  (See Figure~\ref{fig:StandardDiagram} for a 
somewhat complicated Heegaard diagram for $S^3$, ignoring the extra basepoint labelled $z$.)

Inside $\Sym^g(\Sigma)$, there is a pair of $g$-dimensional tori
\[ \Ta=\alpha_1\times\dots\times\alpha_g\qquad{\text{and}}\qquad
\Tb=\beta_1\times\dots\times \beta_g;\] e.g. $\Ta$ is the space of
$g$-tuples of points in $\Sigma$, so that each point lies on some
$\alpha_i$ and no two points lie on the same $\alpha_i$.  
The basepoint gives rise to a real codimension two submanifold
$V_w\subset\Sym^g(\Sigma)$, consisting of those $g$-tuples of points
$\x$ that include the point $w$.

The intersection points $\Ta\cap\Tb$ are called {\em Heegaard states}
for the diagram $\HD$, and they are denoted $\States(\HD)$.
Explicitly, if we think of the $\alpha$- and $\beta$-circles as numbered by $\{1,\dots,g\}$,
then Heegaard states are partitioned according to permutations $\sigma$ on $\{1,\dots,g\}$.
The Heegaard states of type $\sigma$ correspond to  points in the Cartesian product
\[ (\alpha_1\cap \beta_{\sigma(1)})\times\dots\times(\alpha_g\cap\beta_{\sigma(g)}).\]

A complex structure on $\Sigma$ naturally induces a complex structure
on the $g$-fold symmetric product $\Sym^g(\Sigma)$. In fact, the
$g$-fold symmetric product $\Sym^g(\Sigma)$ can be given a K{\"a}hler
structure so that the tori $\Ta$ and $\Tb$ are
Lagrangian~\cite{Perutz}.  Versions of the Heegaard Floer homology of $Y$
correspond to variants of Lagrangian Floer homology for $\Ta$ and
$\Tb$ in $\Sym^g(\Sigma)$, which depend on how one counts
pseudo-holomorphic disks which interact with the subspace $V_w$.

Choosing an almost-complex structure compatible with the symplectic
structure on $\Sym^g(\Sigma)$, one can consider pseudo-holomorphic
disks, as introduced by Gromov~\cite{Gromov}.  For fixed Heegaard states $\x$ and $\y$,
the pseudo-holomorphic
disks in $\Sym^g(\Sigma)$ connecting $\x$ to $\y$
can be organized into homotopy
classes of maps from the unit disk $\CDisk$ in the complex plane
to $\Sym^g(\Sigma)$, $u\colon \CDisk\to \Sym^g(\Sigma)$,
satisfying the following boundary conditions:
$u$ maps $-i$ to $\x$, $i$ to $\y$, and $x+iy=z\in
\partial\CDisk$ with $x\geq 0$ to $\Ta$ and $x\leq 0$ to $\Tb$. We
denote the space of homotopy classes of such maps by $\pi_2(\x,\y)$.
Since $w$ is disjoint from the $\alpha_i$ and $\beta_j$, there is a
well-defined map $n_w\colon \pi_2(\x,\y)\to \Z$ which is given as the
algebraic intersection number of a generic $u$ representing
$\phi\in\pi_2(\x,\y)$ with the oriented submanifold $V_w$.
The moduli space of pseudo-holomorphic disks representing the homotopy class $\phi\in\pi_2(\x,\y)$ is denoted
$\ModFlow(\phi)$. This admits a natural action by $\R$, thought of as the holomorphic automorphisms of $\CDisk$ preserving $\pm i$.

The simplest version of Heegaard Floer homology is the homology of a
chain complex $\CFa(\HD)$, thought of a vector space over the field
$\Field$ with two elements. Generators of this chain complex are the Heegaard states, and its
differential counts pseudo-holomorphic disks that are
disjoint from $V_w$; more formally, $\CFa(\HD)$ is the vector space
generated $\States(\HD)$, equipped with the differential
\[\da (\x)=\sum_{\y\in\States}\sum_{\{\phi\in\pi_2(\x,\y)|n_w(\phi)=0, \Mas(\phi)=1\}}
\#\left(\frac{\ModFlow(\phi)}{\R}\right)\cdot  \y.\]
Here, $\Mas(\phi)$ is the Maslov index of the homotopy class $\phi$~\cite{SpecFlow,FOOO};
see~\cite{LipshitzCyl} for a very useful formulation in terms of the Heegaard diagram.
As is standard in Floer theory~\cite{FloerHoferSalamon,FOOO},
to make sense of the definition, the $\dbar$-equations
need to be perturbed suitably to make the moduli spaces transverse.
This chain complex has a refinement
$\CFm(\HD)$, which is a module over the polynomial algebra $\Field[U]$, 
and whose differential is defined by
\[ \dm(\x)=\sum_{\y\in\States}
\sum_{\{\phi\in\pi_2(\x,\y)|\Mas(\phi)=1\}}
\#\left(\frac{\ModFlow(\phi)}{\R}\right)\cdot U^{n_w(\phi)} \y.\] The
$U=0$ specialization of this chain complex is $\CFa(\HD)$.  (Both
complexes can in fact be defined over $\Z$ coefficients;
see~\cite{HolDisk}.)

The main theorem of~\cite{HolDisk} states that the homology of
$\CFm(\HD)$ (and $\CFa(\HD)$) is an invariant of the underlying
closed, oriented three-manifold $Y$ represented by $\HD$.

Heegaard Floer homology has an extension to knots $K\subset Y$ in a
three-manifold, called {\em knot Floer homology}, which was discovered
independently by Jacob Rasmussen~\cite{RasmussenThesis} and by
us~\cite{HolKnot}. For this version, start with a doubly-pointed
Heegaard diagram $\HD=(\Sigma,\alphas,\betas,w,z)$, where here the two
basepoints $w$ and $z$ in $\Sigma$ are both chosen to be disjoint from
the $\alpha_i$ and the $\beta_i$.  This data specifies an oriented
knot inside the three-manifold $Y$ specified by the Heegaard diagram
$(\Sigma,\alphas,\betas)$. The knot is constructed by the following
procedure. Connect $w$ to $z$ in $\Sigma$ by an arc $a$ that is disjoint
from the $\alpha$-curves, and push the interior of resulting arc into
the $\alpha$-handlebody; similarly, connect $w$ and $z$ by another
arc in $\Sigma$ that is disjoint from the $\beta$-curves and push the interior of that
into the $\beta$-handlebody to get $b$. The knot $K$ is obtained as $a\cup b$.
It can be oriented by the convention that $\partial a = z-w= -\partial b$. 

The simplest version of knot Floer homology is the homology of a chain
complex $\CFKa(\HD)$, once again generated by Heegaard states (in 
a doubly-pointed Heegaard diagram $\HD$ representing $K$), with differential
given by
\[\da_K (\x)=\sum_{\y\in\States}\sum_{\{\phi\in\pi_2(\x,\y)|n_w(\phi)=0=n_z(\phi), \Mas(\phi)=1\}}
\#\left(\frac{\ModFlow(\phi)}{\R} \right)\y.\] For simplicity, we hereafter restrict
attention to the case where the ambient three-manifold is $S^3$.
Dropping the requirement that $n_z(\phi)=0$ gives the chain
complex representing $\HFa(S^3)$, which is a one-dimensional vector
space.  Knot Floer complex is equipped with two gradings, the {\em
Maslov grading} and the {\em Alexander grading}, induced by functions
\[ M\colon \States(\HD)\to \Z \qquad{\text{and}}\qquad
A\colon \States(\HD)\to \Z\]
that are characterized as follows.

The function $M$ satisfies the property that if $\x$ and $\y$ are any
two Heegaard states, and $\phi\in\pi_2(\x,\y)$ is a homotopy class of
Whitney disks, then
\[ M(\x)-M(\y)=\Mas(\phi)-2n_w(\phi).\]
This specifies $M$ uniquely up to an overall additive constant.
The function $M$ induces a $\Z$-valued grading on $\CFa(S^3)$ for which the differential $\da$
drops grading by one; thus there is an induced grading on
$\HFa(S^3)\cong \Field$. The additive indeterminacy on $M$ is pinned
down by requiring that $\HFa(S^3)$ is supported in Maslov grading
equal to zero.

The function $A$ satisfies the property that if $\x$ and $\y$ are any
two Heegaard states, and $\phi\in\pi_2(\x,\y)$ is a homotopy class of
Whitney disks, then
\[ A(\x)-A(\y)=n_z(\phi)-n_w(\phi).\]
Once again, this specifies $A$ up to an overall additive constant; and
the differential $\da_K$ preserves the corresponding splitting of
$\CFKa(\HD)$ specified by the Alexander grading. 

The Maslov and Alexander functions induce a bigrading on $\CFKa(\HD)$,
\[ \CFKa(\HD)=\bigoplus_{d,s\in \Z}\CFKa_d(\HD,s),\]
where $\CFKa_d(\HD,s)$ is generated by those states $\x$ with
$M(\x)=d$ and $A(\x)=s$.  The differential satisfies
\[ \da_K\colon \CFKa_d(\HD,s)\to \CFKa_{d-1}(\HD,s),\]
and therefore the bigrading descends to homology
\[ \HFKa(\HD)=\bigoplus_{d,s\in \Z}\HFKa_d(\HD,s).\]

The bigraded chain complex $\CFKa(\HD)$ has
a {\em graded Euler characteristic}, which is a Laurent polynomial
with integral coefficients, in a formal variable $t$, defined by
\[ \chi(\CFKa(K))=
\sum_{d,s} (-1)^d \dim \CFKa_d(K,s) t^s
=\sum_{d,s} (-1)^d \dim \HFKa_d(K,s) t^s.\]
This graded Euler characteristic coincides with the
Alexander polynomial $\Delta_K(t)$:
\begin{equation}
  \label{eq:Alex}
  \chi(\CFKa(K))\doteq\Delta_K(t),
\end{equation}
where here $\doteq$ means that the two polynomials agree up to overall
factors of $t$~\cite{Knots}.
The additive indeterminacy in $A$ is eliminated by requiring that the graded
Euler characteristic is symmetric in $t$; i.e. Equation~\eqref{eq:Alex}
holds with equality in place of $\doteq$.

The information in the bigraded vector space $\HFKa(K)$ is encoded in its {\em Poincar{\'e} polynomial},
a polynomial with non-negative integral coefficients in two formal variables $q$ and $t$, defined by
\[ P_K(q,t) = \sum_{d,s} \dim \HFKa_d(K,s) q^d t^s.\]
Specializing $P_K$ to $q=-1$ gives the graded Euler characteristic; i.e.   $P_K(-1,t)=\Delta_K(t)$.

The construction described above is analytic in nature: the generators
are combinatorial, but differentials count pseudo-holomorphic
disks. Knot Floer homology has a number of different, more
computationally approachable formulations. We will return to this
point, but first, we outline some properties and applications of the
invariant.

\section{First properties}

We describe now some basic properties of knot Floer homology, contrasting them with corresponding properties
for the Alexander polynomial.

Suppose that $K_+$ and $K_-$ are two knots with a projection that
differs in exactly one crossing, as shown in the first two pictures of
Figure~\ref{fig:SkeinRelation}.  Then, we can resolve the crossing to
obtain a new oriented link with two components, the third picture in
that figure. More generally, if $\orL_+$, $\orL_-$, and $\orL_0$ are
three oriented links that differ as in that figure, we say that they form a
{\em skein triple}. The Alexander polynomial for knots can be extended
to oriented links, and that extension satisfies the following {\em
  skein relation} for any skein triple $\orL_+$, $\orL_-$, $\orL_0$:
\begin{figure} \centering \input{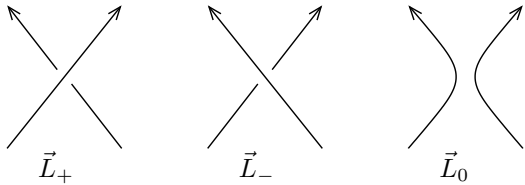}
\caption{{\bf Crossing conventions in the skein relation.}}
\label{fig:SkeinRelation}
\end{figure}
\[ \Delta_{\orL_+}(t)-\Delta_{\orL_-}(t)=(t^{1/2}-t^{-1/2})\Delta_{\orL_0}(t).\]
This relation gives a helpful inductive computational procedure for the Alexander polynomial.
In fact, it was observed by John Conway~\cite{Conway} that the Alexander polynomial (for oriented links) is uniquely characterized
by the above skein relation, and the normalization for the unknot $\Unknot$, 
which states that $\Delta_\Unknot(t)$ is the constant polynomial $1$.
Extending knot Floer homology to links, 
the skein relation has the following analogue:
\begin{thm}
  If $\orL_+$, $\orL_-$, and $\orL_0$ are three oriented links that
  fit into a skein triple, there is a corresponding exact triangle
  relating their bigraded knot Floer homologies.  When the two strands
  meeting at $\orL_+$ belong to the same component of $\orL_+$, the
  triangle has the form \[ \begin{tikzpicture}[x=2.3cm,y=48pt] \node
  at (0,0) (p) {$\HFKa(\orL_+)$} ; \node at (2,0) (n)
  {$\HFKa(\orL_-)$} ; \node at (1,-1) (z) {$\HFKa(\orL_0)$}
  ; \draw[->] (p) to node[above] {} (n) ; \draw[->] (n) to node[below]
  {} (z) ; \draw[->] (z) to node[below] {} (p) ; \end{tikzpicture} \]
  When the strands in $\orL_+$ belong to different components, there
  is a similar triangle, except that $\HFKa(\orL_0)$ is tensored with
  an appropriately graded four-dimensional bigraded vector space.
\end{thm}
        See~\cite{Knots} and~\cite[Chapter~9]{GridBook} for a
precise statement (with specified bigradings).  The above result should be
compared with Floer's exact triangle for instanton homology; compare~\cite{FloerKnots,FloerTriangles,BraamDonaldson}.

The Alexander polynomial polynomial is multiplicative under connected sum.
This has the following generalization to the case of knot Floer homology:

\begin{prop}
  If $K_1$ and $K_2$ are two knots, then
  $\HFKa(K_1\# K_2)$ is obtained as the graded tensor product of the bigraded
  vector spaces
  $\HFKa(K_1)$ and $\HFKa(K_2)$;
  i.e. $P_{K_1\# K_2}(q,t)=P_{K_1}(q,t)\cdot P_{K_2}(q,t)$
\end{prop}

Recall that a knot is called {\em alternating} if it has a diagram
with the property that crossings alternate between over- and
under-crossings as one follows the projection.  By a classical theorem
of Cromwell and Murasugi~\cite{Cromwell,Murasugi}, the Alexander
polynomial of an alternating knot is special: its coefficients
alternate in sign. This has the following analogue for knot Floer
homology~\cite{AltKnots}; see also~\cite{Rasmussen2Bridge,ManolescuOzsvath}.

\begin{thm}
  \label{thm:AltKnots}
  If $K$ is an alternating knot, then the knot Floer homology for $K$
  is determined by its Alexander polynomial $\Delta_K(t)$ and its
  signature $\sigma(K)$, by the formula
  \[ P_{K}(q,t)=q^{\frac{\sigma}{2}}\cdot\Delta_K(q t).\]
\end{thm}

Knot Floer homology can be given more algebraic structure. For example, 
there is a version which is 
a free chain complex over $\Field[U]$, $\CFKm(\HD)$, with differential 
\[\dm_K (\x)=\sum_{\y\in\States}\sum_{\{\phi\in\pi_2(\x,\y)|n_z(\phi)=0, \Mas(\phi)=1\}}
\#\left(\frac{\ModFlow(\phi)}{\R} \right)U^{n_w(\phi)}\y.\] 
Extending the Maslov and Alexander grading so that multiplication by $U$ drops Maslov grading
by $2$ and Alexander grading by $1$, we have that
\[ \dm_K\colon \CFKm_d(\HD,s)\to \CFKm_{d-1}(\HD,s)
\qquad
U\colon \CFKm_d(\HD,s)\to \CFKm_{d-2}(\HD,s-1).\]
Thus, the homology $\HFKm(\HD)$ inherits the structure of a bigraded $\Field[U]$-module.

\begin{prop}
  \label{prop:Structure}
  The bigraded module $\HFKm(K)$ is finitely generated; in fact, it
  consists of direct summands of the form $\Field[U]/U^m$ for various
  choices of $m$, and a single free summand $\Field[U]$.
\end{prop}

The above proposition is clear from the original definitions of knot
Floer homology~\cite{Knots,RasmussenThesis};
see~\cite[Chapter~7]{GridBook} for a more precise reference.

Proposition~\ref{prop:Structure} shows that $\HFKm(K)$ always contains
$U$-non-torsion elements, i.e. elements $\xi\in \HFKm(K)$ with
$U^m\cdot \xi\neq 0$ for all $m$.  Thus, there is a natural numerical
invariant of knots $K$, denoted $\tau(K)$ which is defined as $-1$
times the maximal Alexander grading of any non-torsion class $\xi\in
\HFKm(K)$.

\section{Topological applications}

Knot Floer homology was originally envisioned as a tool for computing
the Heegaard Floer homology groups of three-manifolds obtained as
surgeries on a given knot.  If $K$ is a knot in $S^3$, a ``surgery
formula'' expresses the Heegaard Floer homology $\HFa$ for
three-manifolds obtained as surgeries on $K$ in terms of another
variant of knot Floer homology $\HFKwz(K)$ defined over the ring
$\OurRing=\Field[U,V]/U V = 0$.  This knot invariant is the homology
of a chain complex, $\CFKwz({\mathcal H})$, which is freely generated
(over $\OurRing$) by Heegaard states, and whose differential is given
by
\begin{equation}
        \label{eq:CFKwz}
         \partial(\x)=\sum_{\y\in\States}
\sum_{\{\phi\in\pi_2(\x,\y)|\Mas(\phi)=1\}}
\#\left(\frac{\ModFlow(\phi)}{\R}\right)\cdot U^{n_w(\phi)} V^{n_z(\phi)} \y.
\end{equation}
As in the case of $\HFKm$, the homology module inherits a bigrading: in the present case, $U$ drops
Alexander grading by one, and $V$ raises it by one.  

We do not state the surgery formula here (see~\cite{IntSurg,
RatSurg}), but we do give a consequence:

\begin{thm}
  \label{thm:LSpaceKnots}
  ~\cite{NoteLens}
  Suppose that $K$ is a knot with the property that for some rational number $r\in \Q$,
  the three-manifold $S^3_r(K)$ is a lens space. Then, all the coefficients of
  the Alexander polynomial are $\pm 1$ or $0$; in fact, the non-zero ones alternate in sign.
  Thus, we can write
  \[ \Delta_K(t)=\sum_{k=0}^{n} (-1)^k t^{\alpha_k} \] where $\{\alpha_k\}_{k=0}^n$ is a decreasing sequence of 
  integers. Moreover, when $r>0$,
  the knot Floer homology of $K$ is determined by this Alexander polynomial, 
  as follows. There is a sequence of integers $\{m_k\}_{k=0}^n$ determined by the formulae:
  \begin{align*}
    m_0&= 0 \\
    m_{2k}&=m_{2k-1}-1 \\
    m_{2k+1}&=m_{2k}-2(\alpha_{2k}-\alpha_{2k+1})+1,
  \end{align*}
  so that
  \[ P_{T_{p,q}}(q,t)=\sum_{k=0}^{n} q^{m_k} t^{\alpha_k}.\]
\end{thm}

In particular, for each $s\in \Z$, $\HFKa_*(K,s)$ has dimension $0$ or $1$.

Recall that any knot $K\subset S^3$ can be realized as the boundary of
an orientable surface $F$ embedded in $S^3$.  Such a surface is called
a {\em Seifert surface} for $K$, and the minimal genus of any Seifert
surface for $K$ is called the {\em Seifert genus} of $K$.  It is a
classical result that the degree of the Alexander polynomial gives a
lower bound for the Seifert genus of a knot. This result has a
sharpening for knot Floer homology, which is inspired by work of
Kronheimer and Mrowka~\cite{KMThurston}:

\begin{thm}~\cite{GenusBounds}
        \label{thm:GenusBounds}
  Knot Floer homology detects the Seifert genus $g(K)$ of a knot $K$, in the sense that
  \[ g(K)=\min\{s\big| \HFKa_*(K,s)\neq 0\}.\]
\end{thm}

Our original proof of the above theorem relied on many results in low-dimensional
topology. To start with, Gabai's theory of sutured manifolds equips
the zero-surgery with a taut foliation~\cite{DiskDecomposition}. A theorem of
Eliashberg and Thurston~\cite{EliashbergThurston} provides the product
$[-1,1]\times S^3_0(K)$ with a symplectic structure, which is suitably
convex at the boundary. A theorem of Eliashberg~\cite{Eliashberg} and
Etnyre~\cite{Etnyre} embeds this symplectic cylinder in a closed
symplectic manifold $X$.  Donaldson's Lefschetz
pencils~\cite{DonaldsonLefschetz} on symplectic manifolds then
provides a suitable two-handle decomposition on $X$ for which we can
prove that the Heegaard Floer four-manifold invariant is
non-zero~\cite{HolDiskSymp}. A surgery formula relating knot Floer homology
with the Heegaard Floer homology of the $0$-surgery then gives the
required non-vanishing theorem for knot Floer homology.

Juh{\'a}sz has an elegant proof of the above result~\cite{JuhaszDecomp} that bypasses most
of the above machinery (still building on Gabai's sutured hierarchy),
using his {\em sutured Floer homology}~\cite{Juhasz}.

Theorem~\ref{thm:GenusBounds} has the following  corollary:

\begin{cor}~\cite{GenusBounds}
        \label{cor:DetectUnknot}
  Knot Floer homology detects the unknot, in the sense that 
  $\HFKa(K)$ has dimension one if and only if $K$ is the unknot.
\end{cor}

The above corollary underscores how far knot Floer homology goes
beyond the Alexander polynomial: there are infinitely many knots with
trivial Alexander polynomial. 

Theorem~\ref{thm:LSpaceKnots} has a more precise statement, which
expresses the sequence $\{\alpha_k\}$ concretely in terms of the
surgery coefficient $r$ and the resulting lens space
$L(p,q)$~\cite{NoteLens}. In~\cite{GenusBounds}, we combine this result with
Corollary~\ref{cor:DetectUnknot}, to obtain the following result,
first proved using Seiberg-Witten theory in our joint work with Kronheimer
and Mrowka:
\begin{cor}~\cite{KMOS}
        If $K\subset S^3$ is a knot with the property that $S^3_r(K)\cong {\mathbb{RP}}^3$,
        then $K$ is the unknot.
\end{cor}

See~\cite{Greene} for a vast generalization.

A final property of knot Floer homology motivated by the Alexander
polynomial is based on the classical result that the Alexander
polynomial of a fibered knot is monic.  This has an analogue for knot
Floer homology: if $K$ is a fibered knot with Seifert genus $g=g(K)$,
then
\[ \HFKa_*(K,g)=\bigoplus_{d\in \Z}\HFKa_d(K,g) \]
is one-dimensional~\cite{HolDiskSymp}. This fact has the following remarkable converse, due to
Paolo Ghiggini when $g=1$ and Yi Ni when $g>1$:
\begin{thm}~\cite{Ghiggini, YiNiFibered}
  If 
  $\HFKa_*(K,g)$  is one-dimensional, then $K$ is fibered.
\end{thm}
See also~\cite{JuhaszDecomp}.

Ni's theorem, combined with  Theorem~\ref{thm:LSpaceKnots}, immediately gives the following:
\begin{cor}~\cite{YiNiFibered}
  If $K\subset S^3$ is a knot with the property that $S^3_r(K)$ is a lens space,
  then $K$ is fibered.
\end{cor}

So far, we have focused on applications on the simplest variant of
knot Floer homology, $\HFKa(K)$.  The version $\HFKm(K)$, with its
module structure over $\Field[U]$, has further applications to the
unknotting number and the {\em{slice genus}} of a knot, which we
recall here. Thinking of $S^3$ as a boundary of the four-ball $B^4$,
one can consider {\em slice surfaces}: smoothly embedded surfaces in $B^4$, so
that $F\setminus \partial F$ is mapped to $B^4\setminus \partial B^4=S^3$, and
$\partial F$ is mapped to $K\subset S^3$. The {\em slice genus} of a
knot $K$, denoted $g_4(K)$, is the minimal genus of any slice surface
for $K$. A knot is called a {\em slice knot} if its slice genus is
$0$. Clearly, the Seifert genus of $K$ bounds the slice genus of $K$: $g_4(K)\leq g(K)$.

An {\em unknotting} of $K$ is a sequence of knots $K=K_0,
K_1,\dots,K_n$, where $K_i$ is obtained from $K_{i-1}$ by changing one
crossing, so that $K_n$ is the unknot. The {\em unknotting number} of
$K$, denoted $u(K)$, is the minimal length of any unknotting for
$K$. An $n$-step unknotting for $K$ naturally gives rise to an
immersed surface in $B^4$ with $n$ double points. Resolving these
double-points, we can find a slice surface for $K$ with genus
$n$. This proves the bound $g_4(K)\leq u(K)$.

The module structure $\HFKm(K)$, and specifically the associated
integral invariant $\tau$, gives a lower bound on the slice genus according to
the following:
\begin{thm}
  \label{thm:SliceGenus}
  For any knot $K\subset S^3$, 
  $|\tau(K)|\leq g_4(K)$.
\end{thm}
The above is proved in~\cite{FourBall}; see
also~\cite{RasmussenThesis} for other similar bounds. Sucharit Sarkar gave a
combinatorial proof of Theorem~\ref{thm:SliceGenus} from the perspective of ``grid
diagrams''; see~\cite{Sarkar}; see also~\cite[Chapter~8]{GridBook}.

Direct computation shows that for the $(p,q)$ torus knot $T_{p,q}$, $\tau(T_{p,q})=\frac{(p-1)(q-1)}{2}$.
Thus, Theorem~\ref{thm:SliceGenus} gives another verification of following theorem of Kronheimer and Mrowka, first conjectured by Milnor~\cite{MilnorConjecture}:
\begin{thm}~\cite{KMMilnor}
\label{thm:MilnorConjecture}
For relatively prime integers $p$ and $q$, the torus knot $T_{p,q}$ has
\[ u(T_{p,q})=g_4(T_{p,q})=\frac{(p-1)(q-1)}{2}.\]
\end{thm}
It is easy to see that the quantity appearing in the above theorem
also coincides with the Seifert genus of $T_{p,q}$.  Kronheimer and
Mrowka's proof of the above theorem used Donaldson invariants. A
number of alternative proofs have emerged
since. Rasmussen~\cite{RasmussenSlice} gave the first combinatorial
proof, using the algebraic structure on Khovanov's knot invariants;
see also
see~\cite{Sarkar}.

There are non-orientable analogues of the slice genus, defined as
follows.   Consider possibly non-orientable surfaces $F$
embedded in $B^4$, meeting $S^3$ along $K$, and let $\gamma_4(K)$,
the {\em non-orientable $4$-genus of $K$}
denote the minimal complexity, as measured by the dimension of $H_1(F;\Field)$, for all such choices
of $F$. 
For example, the torus knot $T_{2,2n+1}$ bounds a
$n+\frac{1}{2}$-twisted M{\"o}bius strip, so $\gamma_4(T_{2,2n+1})=1$.

Prior to 2012, the best lower bound on $\gamma_4$ for any knot was $3$.
The situation was vastly improved by the following theorem of Joshua Batson:

\begin{thm}~\cite{Batson}
  \label{thm:Batson}
  The 
  non-orientable $4$-genus can be arbitrarily large; for example,
$\gamma_4(T_{2k,2k-1})=k-1$.
\end{thm}

Batson's proof goes by constructing an explicit surface with stated
complexity, to give an upper bound on $\gamma_4(T_{2k,2k-1})$.
Next,
he gives a lower bound on $\gamma_4(T_{2k,2k-1})$ via a Heegaard Floer
invariant associated to surgeries on the knot. 

An alternative proof of the above theorem is given in joint work of
Andr{\'a}s Stipsicz and the authors~\cite{Unorient}, using another variant of
knot Floer homology. This version is the homology of a chain complex
$\uCFK(\HD)$  which, like $\CFKm(\HD)$, is freely
generated over $\Field[U]$ by the Heegaard states; but it is  equipped
with a differential
\[ \uda\x =\sum_{\y\in \States}\sum_{\{\phi\in\pi_2(\x,\y)\big|\mu(\phi)=1\}}\#\left(\frac{\ModFlow(\phi)}{\R}\right) U^{n_w(\phi)+n_z(\phi)} \y.\]
This complex is equipped with the single grading
$\delta(\x)=M(\x)-A(\x)$. It is straightforward to check that $\uda$
drops $\delta$-grading by $1$, as does multiplication by $U$.
Proposition~\ref{prop:Structure} has the following analogue:
\begin{prop}~\cite[Proposition~3.5]{Unorient}
  The bigraded module $\uHFK(K)$ is finitely generated; in fact, it
  consists of direct summands of the form $\Field[U]/U^m$ for various
  choices of $m$, and a single free summand $\Field[U]$.
\end{prop}
We can now define $\upsilon(K)$ to be the maximal $\delta$-grading of
any $U$-non-torsion element in $\uHFK(K)$.
Theorem~\ref{thm:Batson} can be proved via a computation of $\upsilon(T_{2k,2k-1})$, combined
with the following bound on the non-orientable $4$-genus in terms of $\upsilon$, analogous to 
Theorem~\ref{thm:SliceGenus}:
\begin{thm}~\cite{Unorient}
        For any knot $K\subset S^3$, $|\upsilon(K)-\frac{\sigma(K)}{2}|\leq \gamma_4(K)$.
\end{thm}

Analyzing the slice genus is a place where smooth four-dimensional
topology has a clear interaction with knot theory. The slice surfaces
whose genus is minimized in the definition are thought of as smoothly
embedded in $B^4$. Relaxing this requirement, we could ask for locally
flat, topologically embedded surfaces, to obtain an analogous
numerical knot invariant, called the {\em topological slice genus},
$\gTopSlice(K)$.

In a related direction, one can say that two knots $K_1$ and $K_2$ are
concordant if there is an embedded annulus $F$ in $[1,2]\times S^3$ so
that $F\cap (\{i\}\times S^3)$ is the knot $K_i$; or, equivalently, if
$K_1\# \mirror(K_2)$ is a slice knot, where here $\mirror(K)$ denotes
the mirror of $K$. The connected sum operation endows this set with
the structure of an Abelian group, called the {\em smooth concordance
  group} $\Conc$. If we require the annulus to be only topologically
embedded, or equivalently, if we require $K_1\#\mirror(K_2)$ to be
only topologically slice, we obtain another group, the {\em
  topological concordance group}, denoted $\ConcTop$.  There is a
canonical homomorphism $\Conc\to \ConcTop$, whose kernel is the subgroup of topologically slice knots, 
$\ConcTS$.

Litherland~\cite{Litherland} showed that $\ConcTop$ contains a direct
summand of $\Z^{\infty}$.  According to a theorem of Freedman, any knot
with $\Delta_K(t)=1$ is topologically slice. Using Donaldson's
diagonalizability theorem, Andrew Casson showed that $\ConcTS$ is
non-trivial; see~\cite{CochranGompf}.  In fact, the $0$-twisted
Whitehead double of the trefoil, a knot for which $\tau(K)=1$, gives a
$\Z$-direct summand in $\ConcTS$~\cite{LivingstonComputations}.

Using gauge theory, in 1995 Endo exhibited a $\Z^{\infty}$ subgroup of
$\ConcTS$.  In 2012, Jen Hom~\cite{JenHomInfinitelyGenerated} went
further, exhibiting a $\Z^{\infty}$ direct summand in $\ConcTS$ by
constructing infinitely many linearly independent concordance
homomorphisms from $\ConcTS$ to $\Z$.  Her construction uses an
invariant~$\epsilon$, which can be viewed as derived from knot Floer
homology $\HFKwz(K)$ over $\Field[U,V]/UV$. Using $\epsilon$, she
introduces an equivalence relation on the knot Floer complexes, to form 
a totally ordered Abelian group. The homomorphism are then provided by
the axiom of choice.

In joint work with Stipsicz~\cite{Upsilon}, we constructed another
collection of homomorphisms to $\Z$, using a one-parameter deformation
of knot Floer homology $\tHFK(K)$.  Specifically, for each rational $t=\frac{p}{q}\in
[0,2]$, there is a chain complex $\tCFK(K)$, freely generated over
$\Field[v^{1/q}]$ by Heegaard states,  whose differential now has the form
\[ \dt\x =\sum_{\y\in \States}\sum_{\{\phi\in\pi_2(\x,\y)\big|\mu(\phi)=1\}}\#\left(\frac{\ModFlow(\phi)}{\R}\right) v^{t n_w(\phi)+(2-t)n_z(\phi)} \y.\]
This complex is graded by $\gr_t(\x)=M(\x)-tA(\x)$, so that
multiplication by $v$ drops grading by $1$. When $t=0$, the complex is
independent of the knot, and its homology is simply $\Field[v]$. When
$t=1$, the complex is $\uCFK$ considered above.  Define
$\Upsilon_K(t)$ to be the maximal grading of any $v$-non-torsion
element; in particular, $\upsilon(K)=\Upsilon_K(1)$. Like Hom's homomorphisms, $\Upsilon$
detects $\Z^{\infty}$ direct summands in $\ConcTS$~\cite{Upsilon}; see
~\cite{LivingstonNotes} for an alternative formulation of the invariant
$\Upsilon$ and see~\cite{Chen,FellerParkRay} for further developments.

\section{Heegaard diagrams}

To understand knot Floer homology, it is useful to have 
several possible Heegaard diagrams in hand. The first Heegaard diagram, which we will call the
{\em standard diagram for a knot projection}, is determined as follows.

\subsection{The standard diagram for a knot projection}
Fix a knot projection $\Diag$ for $K$ in ${\mathbb R}^2$, together
with a distinguished edge adjoining the infinite region in the
projection complement. The edge is distinguished by placing a star
somewhere on the edge, as shown on the left in
Figure~\ref{fig:StandardDiagram}.  We call this data a {\em decorated
knot projection} of $K$.  To a decorated knot projection, we can associate a
Heegaard diagram representing $K$, as follows. First, singularize the
projection, so that the crossings are actually double-points. Next,
take a regular neighborhood of the resulting planar graph, to obtain a
handlebody $H$ embedded in ${\mathbb R}^3\subset S^3$. The regions in
the complement of the graph in the plane have two distinguished
regions that adjoin the marked edge, one of which is the infinite
region in $\R^2$. For each bounded region in the graph complement,
there is a corresponding $\alpha$-circle. In a neighborhood of each
crossing, we associate a $\beta$-circle as pictured in
Figure~\ref{fig:StandardDiagram}. Transverse to the marking on the
distinguished edge, we choose also a final $\beta$-circle, again as
shown in Figure~\ref{fig:StandardDiagram}, and place the basepoint $w$
and $z$ on either side of it. Note that $\Sigma$ is oriented as
$-\partial H$.

\begin{figure} \centering \input{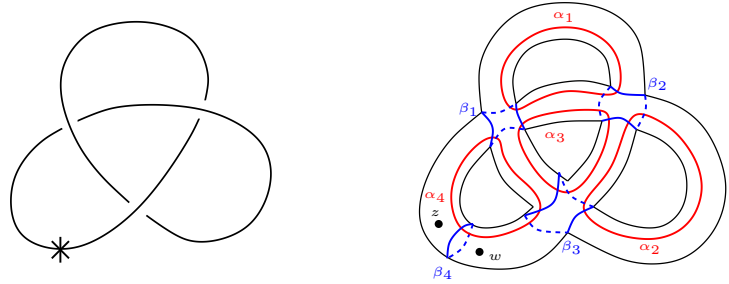}
\caption{{\bf Doubly-pointed Heegaard diagram for the left-handed
trefoil.}}
\label{fig:StandardDiagram}
\end{figure}

We recall here Kauffman's construction of the Alexander polynomial~\cite{Kauffman}.
\begin{defn}
A {\em Kauffman state} for a decorated knot
projection of $K$ is a map $\state$ that associates to each vertex of $G$ one
of the four in-coming quadrants, subject to the following constraints:
\begin{itemize}
\item The quadrants assigned by $\state$ to distinct vertices are subsets of
  distinct bounded regions in ${\mathbb R}^2\setminus G$.
\item The quadrants of the
  bounded region that meets the distinguished edge 
  are not assigned by $\state$ to any of the vertices in $G$.
\end{itemize}
\end{defn}
See Figure~\ref{fig:KauffmanStates} for examples.
\begin{figure}[ht]
\input{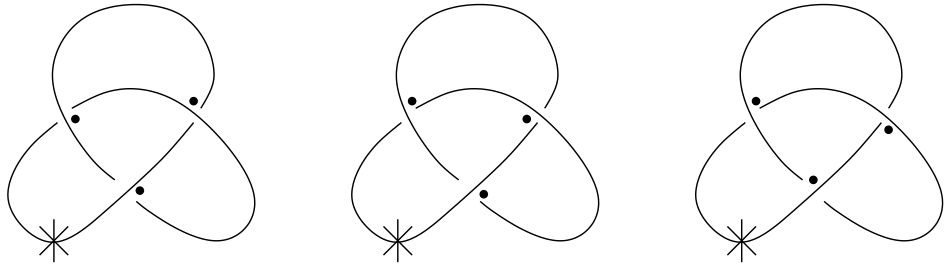}
\caption{\label{fig:KauffmanStates} 
  {\bf{Kauffman states for the left-handed trefoil.}}
   Here all three of the states for this projection.}
\end{figure}

\begin{defn}
  Label the four quadrants about each crossing with $0$, and
  $\pm \OneHalf$, according to the orientations as specified in the
  first line of Figure~\ref{fig:LocalContribs}. The {\em Alexander
  function} of a Kauffman state $\state$, $A(\state)$, is a sum, over
  each crossing, of the contribution of the quadrant occupied by the
  state.  The {\em Maslov function} of a Kauffman state $\state$ is
  obtained similarly, only now the local contributions are as
  specified in the second line of Figure~\ref{fig:LocalContribs}.
\end{defn}

 \begin{figure}[h]
 \centering
 \input{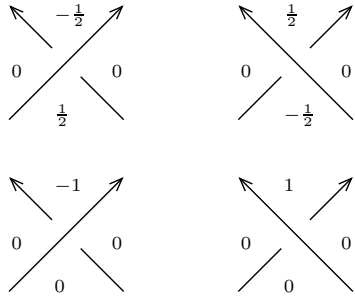}
 \caption{{\bf Local Alexander and
     Maslov contributions.}  The first row illustrates the local Alexander contributions,
and the second the local Maslov contributions of each crossing.}
 \label{fig:LocalContribs}
 \end{figure}

Let $\States=\States(\Diag)$ denote the set of Kauffman states.
Kauffman shows that the  Alexander polynomial is computed by
\[ \Delta_K(t) = \sum_{\x\in\States} (-1)^{M(\x)} t^{A(\x)}.\]
(Note that Kauffman does not define $M(\x)$, which is not needed
for the Alexander polynomial, only its parity.)

To put Kauffman states in even more familiar territory, recall that a
knot projection can be given a checkerboard coloring, coloring each
region on the graph complement black and white so that the two regions
meeting along each edge are colored differently. There is a planar
graph, the ``black graph'', whose vertices correspond to the black
regions in the checkerboard coloring, and whose edges correspond to
crossings in the decorated knot diagram. There is a straightforward one-to-one correspondence
between maximal subtrees in the black graph and Kauffman states; see~\cite{Kauffman}.

The relevance of Kauffman states to knot Floer homology is the following
observation from~\cite{AltKnots}: the Heegaard states in
the standard diagram for a knot projection correspond to the Kauffman
states of the marked projection, via a correspondence which
identifies the corresponding Maslov and Alexander functions.

Although this information is insufficient to compute knot Floer
homology, since the differentials count pseudo-holomorphic disks, it does give
computations in many cases. For example, an elementary argument shows
that for an alternating diagram, $A(\x)-M(\x)$ is independent of the
Kauffman state $\x$.  A little more work shows that for an alternating
knot, $M(\x)-A(\x)=\frac{\sigma(K)}{2}$. Theorem~\ref{thm:AltKnots} is
an immediate consequence of these considerations.
Eun-Soo Lee~\cite{EunSooLee} has shown that a corresponding result also holds for
Khovanov homology~\cite{Khovanov,BarNatan}.

\subsection{$(1,1)$ diagrams} In knot theory, a knot is said to have
a type $(g,b)$ representation if there is a genus $g$ Heegaard
splitting for $S^3$ in which the knot meets each of the two
handlebodies as a union of $b$ standard, unknotted
arcs~\cite{Doll}. Thus, the doubly-pointed Heegaard diagrams described
above give type $(g,1)$ representations of knots.

There is a class of knots for which the Heegaard Floer homology
particularly easy to compute, which can be represented on the torus,
equipped with two basepoints; i.e. which have representations of type
$(1,1)$.  Explicitly, suppose that $\Sigma$ is a surface of genus
$1$, equipped with two basepoints $w$ and $z$, and two curves $\alpha$
and $\beta$ which are isotopic (via an isotopy that crosses $w$ and
$z$) to two curves $\alpha'$ and $\beta'$ that meet transversely in a
single intersection point. Knots with such representations include all
torus knots and all $2$-bridge knots (knots on which there is a height
function with $4$ critical points: $2$ maxima and $2$ minima);
see~\cite{ChoiKo} for a classification.

For a $(1,1)$ knot, the Heegaard Floer homology takes
place in the first symmetric product of the torus $\Sigma$, i.e. within $\Sigma$
itself. Thus, the holomorphic disk counts are combinatorial;
see~\cite{Knots,GodaMatsudaMorifuji}. 

\subsection{Grid diagrams}
A Heegaard diagram representing a $(g,b)$ decomposition can be
represented by a Heegaard surface $\Sigma$ equipped now with $g+b-1$
$\alpha$-curves and $g+b-1$ $\beta$-curves and $2b$ basepoints
$w_1,\dots,w_b$ and $z_1,\dots,z_b$.  The $\alpha$-curves are required
to be pairwise disjoint, and to span a half-dimensional subspace of
$H_1(\Sigma)$; the $\beta$-curves are required to satisfy the same
property. By our homological conditions, the surface obtained by
cutting $\Sigma$ along the $\alpha$-curves has $b$ connected
components.  Our diagrams will satisfy the following
additional property: each of these connected components is required to
have exactly one $w$-basepoint and one $z$-basepoint. Cutting $\Sigma$
along the $\beta$-curves is gives $b$ components, and each component
is required to have exactly one $w$-basepoint and one
$z$-basepoint. This data specifies a three-manifold $Y$ by the natural generalization of the earlier construction:
attach three-dimensional two handles to $ [-1,1]\times \Sigma$ along $\{-1\}\times \alpha_i$
and $\{1\}\times \beta_j$. We are left with a three-manifold whose boundary consists of a collection of two-spheres.
Fill each two-sphere with a three-ball.

We can construct an oriented link in $Y$ that meets $\Sigma$ in
$\ws\cup\zs$, by the following construction.  In each component of
$\Sigma\setminus(\alpha_1\cup\dots\cup\alpha_{g+b-1})$, find an arc
that connects the corresponding $w_i$ and $z_j$, and push that arc
into $[-1,0]\times \Sigma$, so that it meets $\{0\}\times \Sigma$
exactly at $w_i$ and $z_j$.  Find corresponding arcs in
$\Sigma\setminus\betas$, and push those into $[0,1]\times \Sigma$. 
The two types of basepoints give one-to-one correspondences
\[ f_\ws\colon \pi_0(\Sigma\setminus\alphas)\to \pi_0(\Sigma\setminus\betas)
\qquad{\text{and}}\qquad f_\zs\colon \pi_0(\Sigma\setminus\alphas)\to \pi_0(\Sigma\setminus\betas);\]
so $f_{\zs}^{-1}\circ f_{\ws}$ is a permutation of $\pi_0(\Sigma\setminus\alphas)$.
That permutation can be written as a product of cycles; and the number of cycles in the
description gives the number of components of the resulting link.

Heegaard Floer homology has a generalization to this construction, as
well. The ambient symplectic manifold now is $\Sym^{g+b-1}(\Sigma)$, equipped
with two $g+b-1$-dimensional tori $\Ta$ and $\Tb$.  The chain complex
$\CFKm(\HD)$ now is defined over the polynomial algebra
$\Field[U_1,\dots,U_{b}]$, with differential given by
\begin{equation}
  \label{eq:MultipleBasepoints}
 \dm(\x)=\sum_{\y\in\States} \sum_{\{\phi\in\pi_2(\x,\y)|\Mas(\phi)=1, n_{z_1}(\phi)=\dots=n_{z_{b}}(\phi)=0\}}
\#\left(\frac{\ModFlow(\phi)}{\R}\right) U_1^{n_{w_1}(\phi)}\cdots U_{b}^{n_{w_b}(\phi)} \y.
\end{equation}

When the multiply-pointed Heegaard diagram represents a knot $K$, then
all of the $U_i$ variables act the same in homology, and the resulting
$\Field[U]$-module is isomorphic to the bigraded knot Floer homology
$\HFKm(K)$ described earlier; cf.~\cite{MOS, OSlinks}.

This observation is especially powerful for a particular class of
Heegaard diagrams called {\em grid diagrams}, where $\Sigma$ has genus
$1$, all of the $\alpha$-curves are parallel (i.e. isotopic to one
another), and all the $\beta$-curves are parallel. 

It is a classical result that every knot has such a diagram: indeed, a
projection for a knot with $c$ crossings can be turned into a grid
diagram for $K$ with $b=c+2$ $\alpha$-curves and $\beta$-curves.
Moreover, these diagrams are also ``nice'' in the sense introduced by
Sucharit Sarkar. Sarkar showed that for certain Heegaard diagrams, the
holomorphic disk counts appearing in the Heegaard Floer differential have an explicit,
topological formulation~\cite{SarkarWang}. The key
result of Ciprian Manolescu, Sucharit Sarkar, and the second author
in~\cite{MOS} states that the holomorphic disk counts appearing in
Equation~\eqref{eq:MultipleBasepoints} for grid diagrams is a
combinatorial count of certain embedded rectangles in the Heegaard
torus. In~\cite{BaldwinGillam}, these techniques are used to compute
the knot Floer homology groups of knots with $\leq 12$ crossings.

The resulting chain complex, whose generators correspond to
permutations and whose differential counts embedded rectangles, can
be taken as a definition for the theory rather than a computation.
Invariance can be formulated and proved within the realm of grid
diagrams: there is a well understood set of moves that connect any two
grid diagrams representing the same knot~\cite{Cromwell,Dynnikov}. One
can construct isomorphisms between the corresponding ``grid homology
groups'', to show that the result is knot invariant.  This is the
approach taken in~\cite{MOST}; see also~\cite{GridBook}.

The basic setup of grid homology requires little machinery: gone are
the pseudo-holomorphic curves, replaced instead by embedded
rectangles.  This makes the material perhaps more accessible to
students trying to enter the subject.  The perspective offered by
grids naturally points to further applications, especially to
Legendrian knot theory~\cite{OST,NOT}; see
also~\cite[Chapter~12]{GridBook}. Moreover, some of the topological
applications have proofs purely within the framework of grid diagrams.
As pointed out earlier, the slice genus bounds have a combinatorial
formulation (see Theorems~\ref{thm:SliceGenus}
and~\ref{thm:MilnorConjecture} above). Some non-orientable $4$-genus
bounds (see Theorem~\ref{thm:Batson}) have combinatorial
proofs~\cite{Unorient}.  

Working entirely in the world of grid diagrams does have some
disadvantages, though.  At present, many of the topological
applications cannot be understood from the grid perspective. More
frustratingly, the chain complexes associated to grid diagrams tend to
be large and unwieldy. For a knot represented by an $n\times n$
grid diagram, the grid chain complex has $n!$ generators.

Much has been written on the topic of grid diagrams, so we refer the
interested reader to the literature. We will focus instead on a
different more algebraic computational
approach~\cite{BorderedKnots,HFKa2,HolKnot}, motivated by ``bordered
Floer homology''~\cite{InvPair}.

\section{Bordered Preliminaries}
Bordered Floer homology is an invariant for three-manifolds with
boundary introduced in 2008, Robert Lipshitz, Dylan Thurston, and
the second author~\cite{InvPair,PNAS}. 
This theory associates a differential graded algebra to a
surface $F$ equipped with a parameterization, $\Alg(F)$.  To an
oriented three-manifold $Y_1$, equipped with an identification $F\cong \partial Y_1$,
the bordered theory associates an $\Ainfty$ module over this algebra,
denoted $\CFAa(Y_1)$. For an oriented three-manifold $Y_2$ whose
boundary is identified with $-{\widehat F}$, the theory associates an
algebraic object, called a ``type $D$ structure'' $\CFDa(Y_2)$, over
$\Alg(F)$, which can be thought of as a kind of
free differential module over $\Alg(F)$.  The module operations are
defined by certain pseudo-holomorphic disks occurring in naturally
adapted Heegaard diagrams that represent bordered three-manifolds.  
We recall here some of the formal aspects of this
theory, as they serve as a motivation for some algebraic constructions
for knot Floer homology which we will describe later.

As a preliminary point, recall that differential graded algebra $\Alg$ is a graded vector space
$\Alg$ equipped with an associative multiplication and a differential, which are compatible by the Leibniz rule
\[ d(a\cdot b)=(da)\cdot b + a\cdot (db).\] We suppress signs here, as
we are working with coefficients in $\Zmod{2}$.  
Sometimes the differential and the multiplication are
denoted by the more uniform notation
\[ \mu_1\colon \Alg\to \Alg\text\qquad{\text{and}}\qquad \mu_2\colon \Alg\otimes\Alg\to \Alg.\]
Then, the structure relations are
$\mu_1\circ \mu_1 = 0$, $\mu_2(\mu_2(a,b),c)+\mu_2(a,\mu_2(b,c))=0$ (associativity), and
$\mu_1(\mu_2(a,b))=\mu_2(\mu_1(a),b)+\mu_2(a,\mu_1(b))$.

Differential graded algebras have a natural generalization to $\Ainfty$ algebras~\cite{Keller}, which are graded vector spaces
$\Alg$ equipped with a sequence of maps
\[ \{\mu_n\colon \Alg^{\otimes n}\to \Alg\}_{n=1}^{\infty},\]
satisfying an infinite sequence of structure relations (generalizing
the three structure relations for differential graded algebras stated
above), called the {\em $\Ainfty$ relations}. To state these, it is
useful to think of planar trees $T$, with $k$ inputs and one output.
Each such tree gives rise to a map $\mu(T)\colon \Alg^{\otimes k}\to
\Alg$, where each vertex with valence $d$ is labelled by the operation
$\mu_{d-1}$. The $\Ainfty$ relation with $k$ inputs, states the sum of
$\mu(T)$, taken over all trees $T$ with $k$ inputs and exactly two
internal vertices, vanishes.  For example, there is a single tree with
two internal vertices: it is the linear tree with two valence two
vertices. So the $\Ainfty$ relation in this case states that
$\mu_1\circ\mu_1=0$. A more interesting example is shown in
Figure~\ref{fig:AinftyRelation}.  From this perspective, a
differential graded algebra is an $\Ainfty$ algebra with $\mu_n=0$ for
all $n\geq 3$.
\begin{figure} \centering \input{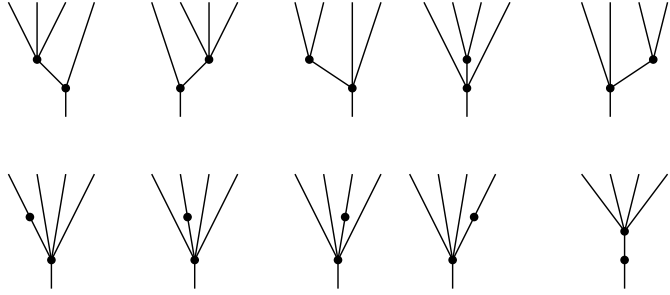}
\caption{{\bf The $\Ainfty$ relation with four inputs.}
The sum of the operations associated to these trees vanishes; for example, 
the tree on the top left contributes $\mu_2(\mu_3(a_1,a_2,a_3),a_4)$.}
\label{fig:AinftyRelation}
\end{figure}

The $\Ainfty$  relations can alternatively be formalized as follows. Consider the ``bar complex'', the vector space
\[ {\mathrm{Bar}}(\Alg) = \bigoplus_{i=1}^{\infty} \Alg^{\otimes i}, \]
equipped with the endomorphism
\begin{equation} 
\label{eq:dbar}
\dbar(a_1\otimes \dots\otimes a_n) =\sum_{r\geq 0,s>0, r+s\leq n} (a_1\otimes\dots\otimes a_r)
   \otimes \mu_{s}(a_{r+1}\otimes\dots \otimes a_{r+s})\otimes (a_{r+s+1}\otimes\dots\otimes a_n).
\end{equation}
The $\Ainfty$ relation is equivalent to the condition that $\dbar\circ\dbar=0$.

Over a differential graded algebra $\Alg$, it is natural to consider differential graded modules $M$,
which are equipped with a differential $m_1\colon M\to M$, and an associative action $m_2\colon M\otimes \Alg\to M$.
These objects have a natural $\Ainfty$ generalizations:
an $\Ainfty$ module $M$ is a graded vector space equipped with a sequence of maps
\begin{equation}
        \label{eq:DefAct}
        \{m_n\colon M\otimes \Alg^{\otimes(n-1)}\to M\}_{n=1}^{\infty}.
\end{equation}
Again, these are required to satisfy an $\Ainfty$ relation, which is exactly as in the case for algebras, with the
understanding that now all trees $T$ have a distinguished leftmost strand (corresponding to $M$), along which
all vertices are labelled with $m_i$, rather than $\mu_i$ which labels all other vertices.

Stasheff~\cite{Stasheff} introduced $\Ainfty$ algebras in his study of
algebraic topology for $H$-spaces.  They have since resurfaced in a
number of settings: for example, they have taken a central role in
symplectic geometry~\cite{KontsevichICM,SeidelBook}; they have found
applications in gauge theory~\cite{KMOS,Bloom,KMKhovanov}; and of
course they are also at the heart of bordered Floer homology. Although
we will not need $\Ainfty$ algebras in our subsequent discussions, we
will be considering $\Ainfty$ modules over differential graded
algebras.

Next, we recall the notation of a  $D$ structure over a differential graded
algebra $\Alg$, which is another key player in the bordered theory.
A type $D$ structure  is a graded vector space $X$, equipped with a map
\begin{equation}
        \label{eq:DefTypeD}
         \delta^1\colon X \to \Alg\otimes X 
\end{equation}
satisfying a structure equation
\[ (\mu_2\otimes \Id_X)\circ (\delta^1\otimes \Id_X) \circ \delta^1  + 
(\mu_1\otimes \Id_X)\circ \delta^1 = 0.\]
More concretely, if $X$ has a basis $\{\x_i\}_{i=1}^n$, we can write
\[ \delta^1(\x_i)=\sum_{j=1}^n a_{i,j}\otimes \x_j,\]
for $a_{i,j}\in \Alg$. The structure relation takes the form
\[ d a_{i,k} + \sum_{j} a_{i,j}\cdot a_{j,k}=0.\]

There is a natural pairing between $\Ainfty$ modules $M$ and type $D$
structures $X$ over $\Alg$~\cite{InvPair}, denoted $M\DT X$, defined as follows.
Iterate $\delta^1$ to define a map
\[ \delta^j\colon X \to \Alg^{\otimes j}\otimes X.\]
More precisely, define $\delta^j$ inductively 
by $\delta^0=\Id_{X}$, and \[ \delta^j=
(\Id_{\Alg^{\otimes (j-1)}}\otimes \delta^1) \circ \delta^{j-1}\] for $j>0$; e.g.
\begin{align*}
\delta^2(x)& =(\Id_{\Alg}\otimes \delta^1)\circ\delta^1 \\
\delta^3(x)&=(\Id_{\Alg\otimes\Alg}\otimes \delta^1)\circ 
(\Id_{\Alg}\otimes \delta^1)\circ \delta^1.
\end{align*}
Equip the vector space $M\otimes X$ with the endomorphism
\[ D(p\otimes x)= \sum_{j=0}^{\infty} (m_{j+1}\otimes \Id_X)\circ
\delta^j(x).\] In general, the sum defining $D$ may not be finite;
but there are some instances where it is. For example, the module $M$ is said
to be {\em algebraically bounded} $m_j=0$ for all $j$ sufficiently
large; and a type $D$ structure $X$ is said to be {\em algebraically
  bounded} if $\delta^j=0$ for all $j$ sufficiently large. Boundedness
of either structure is sufficient to ensure finite sums in the
definition of $D$.

In cases where $D$ is well defined, $D^2=0$; i.e. $(M\otimes X,D)$ is a chain complex.
This chain complex is
denoted $M\DT X$ and it agrees with the derived tensor product of the two
$\Ainfty$ modules underlying $M$ and $\Alg\otimes X$; see~\cite{InvPair}.

A key property of bordered Floer homology is a {\em pairing theorem},
which, for a three-manifold $Y$ decomposed as $Y=Y_1\cup_{F} Y_2$,
expresses $\HFa(Y)$ in terms of the above pairing between the type $D$ and the type $A$ structures of the pieces,
$\HFa(Y)\simeq H(\CFAa(Y_1)\DT \CFDa(Y_2))$.

Bimodules have a natural generalization to the $\Ainfty$
setting. Informally, if $\Alg_1$ and $\Alg_2$ are differential graded
algebras, a type $DA$ bimodule $\lsup{\Alg_1}X_{\Alg_2}$ is an object which can be viewed as 
a type $D$ structure over $\Alg_1$, but it also has higher operations
\[ \delta^1_{i+1}\colon X \otimes \Alg_2^{\otimes i}\to \Alg_1\otimes X,\]
satisfying an appropriate $\Ainfty$ relation~\cite{Bimodules}.

Bimodules play the following role in the bordered theory. Recall that
modules associated to a three manifold depend on the boundary
parameterization.  To each mapping class $\phi\colon F\to F$ there is
a corresponding bimodule with the property that if $Y_1'$
is obtained by composing the boundary parameterization of $Y_1$ with
$\phi$, then $\CFDa(Y_1')$ is the tensor product of $\CFDa(Y_1)$ with
the associated bimodule; see~\cite{Bimodules}.

Bordered Floer homology can be used to effectively compute $\HFa(Y)$
(with $\Field$ coefficients). The key point is that the bimodules
associated to mapping class group generators can be computed
explicitly~\cite{HFa}.  Thus, if we start from a Heegaard decompositon
of $Y$, thought of as a union of two standardly framed handlebodies,
glued via an identification $\phi$, which is expressed as a product of
the mapping class group generators, then $\HFa(Y)$ can be obtained as
an iterated tensor product, where the two outermost factors are the
modules associated to the standard handlebodies, and the inner factors
are the bimodules associated to the mapping class group generators appearing in
the factorization of $\phi$. Conversely, this description can be taken
as the definition of $\HFa(Y)$, and its topological invariance
properties can be verified by some model computations. This is the
perspective pursued in work of Bohua Zhan~\cite{Zhan}.

We will describe next an analogous bordered formulation for computing
knot Floer homology; compare also~\cite{Zarev, PetkovaVertesi}.

\section{Bordered algebras and knot invariants}

Bordered knot Floer homology,
defined in~\cite{BorderedKnots} and~\cite{HFKa2},
is a technique for computing knot Floer
homology, which can be thought of as obtained from slicing a decorated
knot projection along horizontal slices.
Specifically, cut the decorated knot projection into slices
$y=t_1<\dots<t_m$, with the following properties:
\begin{itemize}
\item 
the portion of the diagram with $y\leq t_1$ consists of a single strand with 
the global minimum on it
\item the portion with $y \geq t_m$ consists of a single strand with
  the global maximum on it.
\item 
each portion of the diagram
with $t_i\leq y \leq t_{i+1}$ is one of the following three
standard pieces: a local maximum, a local minimum, or a crossing.
\end{itemize}
To each $y=t_i$ slice of the diagram, we will associate an algebra. To
each standard piece we associate a bimodule over the two algebras
associated to its boundary.  A chain complex computing the invariant is then obtained by
tensoring together all of these bimodules.  These generators
correspond to Kauffman states; and indeed generators of the
intermediate bimodules correspond to certain ``partial'' Kauffman
states.  The homology of the resulting chain complex is a knot
invariant.  We describe these ingredients in a little more detail
presently.

\subsection{Partial knot diagrams}

For generic $t$, a decorated knot projection $\Diag$ in the $(x,y)$
plane meets the line $y=t$ in $2n$ transverse points.  We will draw
our diagram so that the distinguished star is the global minimum $y_0$ of
the function $y$ restricted to the projection.

The portion of the diagram contained in the half-space 
in $y\geq t$, for generic $t>y_0$,  is called a {\em upper knot diagram}.

Fix an upper diagram, and suppose that it meets the $y=t$ slice at the
$2n$ points $\{(i,t)\}_{i=1}^{2n}$. These intersection points divide
the $y=t$ line into $2n+1$ connected components
$J_0=(-\infty,1),J_1=(1,2), \dots,J_{2n-1}=(2n-1,2n),J_{2n}=(2n,\infty)$. An
{\em idempotent state} $\x$ is an $n$-element subset of $\{0,\dots,2n\}$. 

An {\em upper Kauffman state} for an upper diagram $y\geq t$ is a pair
$(\state,\x)$ where $\state$ is a  function
that associates to each crossing in the upper diagram
one of the four adjacent quadrants, and $\x$ is an idempotent state for
the $y=t$ slice of the diagram, subject to the following constraints:
\begin{itemize}
\item The quadrants assigned by $\state$ to different crossings are
  subsets of distinct bounded regions in the upper diagram. (The
  regions which contain quadrants assigned by $\state$ are called
  {\em occupied}.)
    \item The unbounded region meets none of the intervals in $\x$
 \item Each bounded, unoccupied bounded region contains exactly one of the intervals
  appearing in $\x$ on its boundary.
\end{itemize}

\begin{figure}[h] 
                  \centering \input{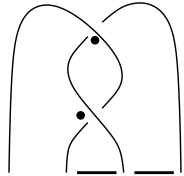} \caption{{\bf An
upper Kauffman state.} The black dots represent the quadrants assigned
by $\kappa$; the dark intervals on the bottom represent $\x$. This
upper diagram has five states.}  \label{fig:UpperDiagTref} \end{figure}

Note that any Kauffman state can be restricted to an upper diagram to
give an upper Kauffman state. 

Upper Kauffman states have the following generalization.  A {\em
  partial knot diagram} is a portion of a knot diagram contained in
the $(x,y)$ plane with $t_2\leq y \leq t_1$, so that $t_1$ and $t_2$
are generic.

A {\em partial Kauffman state} is a triple of data $(\state,\x,\y)$, where
$\x$ is a collection of components in the $y=t_2$ slice, 
$\y$ is a collection of components in the $y=t_1$ slice, and $\state$
is a map that associates to each crossing one of its four adjacent
regions, subject to certain constraints. 
\begin{itemize}
        \item Distinct crossings are assigned by $\state$ to quadrants contained in distinct regions
        in the partial diagram.
      \item If $R$ is occupied, then $\y$ contains all the intervals
        in $R\cap (y=t_1)$ and $\x$ contains none of the intervals in
        $R\cap (y=t_2)$.
      \item If $R$ is unoccupied, then either
            $R$ meets the $y=t_1$ slice,  $\y$ contains all but
        one of the edges of $R\cap(y=t_1)$,  and $\x$ contains none
        of the intervals in $R\cap (y=t_2)$; 
        or $\y$ contains all of the intervals in $R\cap(y=t_1)$ (which now can be empty)
         and $\x$ contains exactly one of the intervals
        in the slice $R\cap(y=t_2)$.
\end{itemize}

\begin{example}
\label{ex:TrivialDiag}
  Consider the partial knot diagram consisting of $2n$ vertical lines.
  In this partial diagram, the partial Kauffman states
  $(\state,\x,\y)$ have $\x=\y$, an arbitrary $n$-element subsets of
  $\{0,\dots,2n\}$; and $\state$ has no information (as there are no
  crossings).
\end{example}

\begin{example}
  \label{ex:LocalMax} Consider the partial knot diagram consisting of
  $2n$ vertical lines, and a single additional strand which contains a
  local maximum, so that two of the strands meet the bottom; see
  Figure~\ref{fig:LocalMaximum}.  Assume that the maximum does not
  appear in the unbounded region. Then, there is a region $R$ in the
  diagram that meets the top boundary in its $(c-1)^{st}$ interval,
  and it meets the bottom boundary its $(c-1)^{st}$ and $(c+1)^{st}$
  intervals.  The partial Kauffman state is then uniquely determined
  by $\x$, which necessarily contains $c$. There are three cases; a
  state is said to be of type $\XX$ if
  $\x\cap \{c-1,c,c+1\}=\{c-1,c\}$, it is of type $\YY$ if
  $\x\cap \{c-1,c,c+1\}=\{c,c+1\}$, and $\ZZ$ if
  $\x\cap \{c-1,c,c+1\}=\{c\}$.  
  See Figure~\ref{fig:LocalMaximum}.
\end{example}

\begin{figure}[h] 
                  \centering \input{MaximumDiag.pstex_t} \caption{{\bf
Partial Kauffman states for the local maximum.} We have drawn here three partial Kauffman states, one of each type.}  
\label{fig:LocalMaximum} \end{figure}

\begin{example}
        \label{ex:Crossing}
  Consider the partial knot diagram consisting of $2n$ strands drawn
  so that the $i^{th}$ and $(i+1)^{st}$ cross exactly once.  There are
  four kinds of partial Kauffman states, according to which of the
  four regions is assigned to the crossing: $\NN$, $\SS$, $\EE$, or
  $\WW$.  For the crossing of type $\NN$, $i\in \x\cap \y$ and
  $\x=\y$; for a crossing of type $\SS$, $i\not\in \x\cup \y$ and
  $\x=\y$; for a crossing of type $\WW$, $i-1\in \y$, $i\not\in \y$,
  $i-1\not\in\x$ and $i\in \x$, and $\y\setminus\{i-1\}=\x\setminus\{i\}$;
  for a crossing of type $\EE$, $i+1\in\y$, $i\not\in \y$,
  $i+1\not\in\x$, $i\in \x$, and $\y\setminus\{i+1\}=\x\setminus\{i\}$.
\end{example}

\begin{figure}[h] 
                  \centering \input{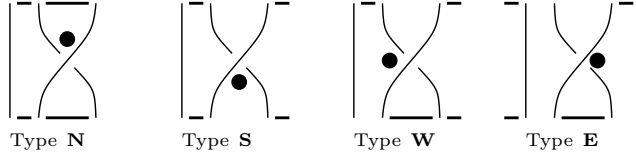} \caption{{\bf
Partial Kauffman states for crossings.} We have drawn here three four Kauffman states, one of each type.}  
\label{fig:CrossingProj} \end{figure}

\subsection{Algebras}
We explain how to associate an algebra to each horizontal slice of a knot diagram. The
horizontal slice can be thought of as a collection of $2n$ points on
the real line. The portion of the knot projection above this
horizontal slice gives a pairing between the $2n$
points. Specifically, if we slice the projection at the line $y=t$, so
that the knot meets the line in points $\{(i,t)\}_{i=1}^{2n}$, then
$i$ and $j$ are matched if $(i,t)$ and $(j,t)$ are joined by an arc in
the diagram contained in the portion of the diagram where $y\geq
t$. We denote this data by $\Matching$.  We will now define the
corresponding algebra $\Alg(n,\Matching)$.

As a preliminary step, we define an algebra $\Blg_0(m,k)$ associated
to $m$ points $\{1,\dots,m\}$ and an integer $0\leq k\leq m+1$;
see~\cite{BorderedKnots}.  The algebra is defined over the polynomial
algebra $\Field[U_1,\dots,U_m]$; and it is is equipped with a set of
preferred idempotents, which correspond to $k$ element subsets of
$\{0,\dots,m\}$ or, equivalently, increasing functions
$\x\colon \{1,\dots,k\}\to \{0,\dots,m\}$ called {\em idempotent
states}.  Let $I(\x)$ denote the idempotent corresponding to the
idempotent state.

Given any idempotent states $\x,\y$, the $\Field[U_1,\dots,U_m]$-module
$I(\x)\cdot \Blg_0(m,k)\cdot I(\y)$ is isomorphic to 
$\Field[U_1,\dots,U_m]$;i.e. it is given with a preferred
generator, which we denote $\gamma_{\x,\y}$.

Thus, for any $\x,\y,\z$, $\gamma_{\x,\y}\cdot \gamma_{\y,\z}=
P_{\x,\y,\z} \cdot \gamma_{\x,\z}$ for some
$P_{\x,\y,\z}\in \Field[U_1,\dots,U_m]$.
To specify the multiplication
on $\Blg_0(m,k)$, it suffices to specify the $P_{\x,\y,\z}$
for all triples of idempotent states, which we do as follows.
Each idempotent state $\x$ has a {\em weight vector} $v^{\x}\in \Z^m$, with components given by
\[ v^{\x}_i=\# \{x\in \x\big| x\geq i\}.\]
Let $P_{\x,\y,\z}$ to be the monomial in $\Field[U_1,\dots,U_m]$,
$U_1^{n_1}\cdots U_m^{n_m}$, where $n_i$ is given by 
\[         n_i =  \OneHalf(|v^{\x}_i - v^{\y}_i| +|v^{\y}_i-v^{\z}_i|-|v^{\x}_i-v^{\z}_i|)
\]

Let $L_i \in \Blg_0$ with $i\in \{1,\dots,m\}$
be the sum of the generators $\gamma_{\x,\y}\in I_\x\cdot \Blg_0\cdot I_\y$ taken over all pairs of idempotent states 
$\x$ and $\y$ with the property that
\[ v^\x_j-v^\y_j=\left\{\begin{array}{ll}
1 &{\text{if $i=j$}} \\
0 &{\text{otherwise.}}
\end{array}\right.\]
Similarly, define $R_i$ to be the sum of all the elements
$\gamma_{\y,\x}\in I_\y\cdot \Blg_0  \cdot I_\x$,
where $\x$ and $\y$ run over all idempotent states as above.

Let $\Blg(m,k)$ be the quotient algebra of $\Blg_0(m,k)$ by the relations
\[ L_{i+1}\cdot L_i = 0,\qquad R_i \cdot R_{i+1}=0 \]
and
\[ I_{\x}\cdot U_j = 0 \] if $\x^{-1}(\{j-1,j\})=\emptyset$; i.e. if
$\Ideal$ denotes the two-sided ideal generated by $L_{i+1}\cdot L_i$,
$R_i\cdot R_{i+1}$, and $\Idemp{\x}\cdot U_j$ as above, then $\Blg=\Blg_0/\Ideal$.

We define a differential graded algebra
$\Alg(n,\Matching)$ by adjoining $n$ central elements $C_{i,j}$ to
$\Blg(2n,n)$, one for each $\{i,j\}\in\Matching$, which satisfy the
relation $C_{i,j}^2=0$. We introduce a differential $d$ on
$\Alg(n,\Matching)$ by $d C_{i,j}=U_i U_j$. 

The algebras can be given gradings, after choosing an orientation on $K$; see~\cite{HFKa2}.

\begin{example}The algebra $\Blg(2,1)$ has the following geometric
  description.  Consider the graph with three vertices, labelled
  $0$, $1$, and $2$, and two edges, one connecting $0$ to $1$ and another
  connecting $1$ and $2$. Think of the path as drawn
  horizontally, so that $1$ is to the right of $0$.  The algebra
  $\Blg(2,1)$ can be thought of as the quotient of the algebra of all
  paths in this graph, and obtained by dividing out by all paths that
  connect $0$ and $2$.  The constant paths at $0$, $1$, and $2$
  correspond to the three idempotents $\Idemp{\{0\}}$,
  $\Idemp{\{1\}}$, and $\Idemp{\{2\}}$; the first edge corresponds to
  $L_1$ and $R_1$ (depending on the direction it is crossed), and the
  second corresponds to $L_2$ and $R_2$.  Here, $U_1=L_1 \cdot R_1 +
  R_1\cdot L_1$ and $U_2=L_2\cdot R_2 + R_2\cdot L_2$. Clearly, the
  relation $U_1 U_2=0$ holds in $\Blg(2,1)$.

To construct $\Alg(1,\{1,2\})$, we adjoin one variable $C_{\{1,2\}}$ whose square is zero.
We think of this as differential graded algebra, but the differential is identically zero.
In particular, 
\[ \Idemp{\{1\}}\cdot  \Alg(1,\{1,2\})
\cdot\Idemp{\{1\}}
\cong \frac{\Field[U_1,U_2,C_{\{1,2\}}]}{(U_1 U_2, C_{\{1,2\}}^2)}.\]
\end{example}

We will view $\Alg=\Alg(n,\Matching)$ as an algebra over the base ring of idempotents 
\[\IdempRing(\Alg)\cong \Field^{{2n+1}\choose{n}}.\]
As such, an $\Ainfty$ module over $\Alg$ will be a right module over
$\IdempRing(\Alg)$, and the actions $m_n$ will be multilinear over
$\IdempRing(\Alg)$: i.e. the tensor products in
Equation~\eqref{eq:DefAct} are taken over that ring. Similarly, a type
$D$ structure is a left module over $\IdempRing(\Alg)$, the map
$\delta^1$ is a $\IdempRing(\Alg)$-module map, and the tensor product
appearing in Equation~\eqref{eq:DefTypeD} is also over that
ring. Moreover, in the definition of $M\DT X$, the underlying vector
space is $M\otimes_{\IdempRing(\Alg)} X$.

\subsection{Bimodules}
Having defined the algebra, we must associate bimodules to the various pieces.

First, consider the module associated to the global maximum. Here,
there is a single upper Kauffman state, and the corresponding
generator $\z$ has
\begin{align*}
        \z &= \Idemp{\{1\}}\cdot \z \\
         \delta^1 \z  &= C_{1,2}\otimes \z.
\end{align*} 

We can think of the tensor products in the construction of
$\KC(\Diag)$ as an iterative procedure, starting with this type $D$
structure as a first step, and then successively increasing the size
of the diagram covered by tensoring the type $D$ structure in hand
with the DA structure associated to the partial knot diagram
immediately below it.  As we will indicate below, the generators of
this type $D$ structure correspond to upper Kauffman states
$(\state,\x)$ for the diagram, and whose left idempotent is
$\Idemp{\x}$.  

Thus, a key ingredient going into this definition is the type $DA$ bimodule associated to each standard
partial diagram.

We do not describe the bimodules explicitly here; we refer the
interested reader to~\cite{HFKa2}. Instead, we explain how to specify them uniquely up to homotopy equivalence.

To this end, it is useful to make the following observations. 
Consider the dual complex for the bar complex, i.e. 
\[ \Cobar(\Alg) = \bigoplus_{i=1}^{\infty} \Hom(\Alg^{\otimes i},\Field), \]
equipped with a differential which is hom dual to $\dbar$ as given in Equation~\eqref{eq:dbar}.
This is an algebra, with multiplication induced by the natural map 
\[ \Hom(\Alg^{\otimes i},\Field)\otimes \Hom(\Alg^{\otimes j},\Field)\to \Hom(\Alg^{\otimes (i+j)},\Field).\]
A bounded $\Ainfty$ module over a differential graded algebra $\Alg$
is the same thing as a type $D$ structure over $\Cobar(\Alg)$.
More generally, a (bounded) $DA$ bimodule, $\lsup{\Alg_1}X_{\Alg_2}$
is the same thing as a type $D$ structure over the tensor product algebra
$\Alg_1\otimes\Cobar(\Alg_2)$. 

The algebra ${\mathrm{Cobar}}$ tends to be rather large; so instead,
it is often convenient to find a smaller quasi-isomorphic version
$\Alg'$. Then, up to quasi-isomorphism, a bounded $\Ainfty$ module over
$\Alg$ is equivalent to a type $D$ structure over $\Alg'$. (This
equivalence is called {\em Koszul duality}~\cite{HomPair}; see also~\cite{KDual}.)
Similarly, a suitably bounded type $DA$ bimodule $\lsup{\Alg_1}X_{\Alg_2}$ is uniquely determined
(up to quasi-isomorphism) by a corresponding module $\lsup{\Alg_1\otimes\Alg_2'}Y$. Correspondingly, 
type $D$ structures over $\Alg_1\otimes \Alg_2'$ are called {\em type $DD$ bimodules} over $\Alg_1$ and $\Alg_2'$.

There is a handy Koszul dual algebra to $\Alg(n,\Matching)$, denoted
$\Alg'(n,\Matching)$, defined as follows.  This is defined
over the base algebra $\Blg(2n,n+1)$, only now we adjoin $2n$
variables $E_1,\dots,E_{2n}$ which satisfy the following relations:
$E_i \cdot E_j = E_j \cdot E_i$ if $i$ and $j$ are not matched,
and $d E_i = U_i$. 

The identity map from $\Alg(n,\Matching)$ to itself can be thought of
as a type $DA$ bimodule over $\Alg(n,\Matching)$; which we can think of as the 
bimodule associated to the trivial diagram from Example~\eqref{ex:TrivialDiag}.
This is Koszul dual
to the type $D$ structure $\CanonDD$ over
$\Alg(n,\Matching)\otimes\DuAlg(n,\Matching)$ whose generators are
$\Idemp{\x}\otimes\Idemp{\y}$, where $\x$ and $\y$ are complementary
idempotent states; i.e. $\x\cup\y=\{0,\dots,2n\}$.
The differential is specified by the element
\[
A = \sum_{i=1}^{2n} \left(L_i\otimes R_i + R_i\otimes L_i\right) + \sum_{i=1}^{2n}
U_i\otimes E_i + 
\sum_{\{i,j\}\in\Partition} C_{\{i,j\}}\otimes \llbracket E_i,E_j\rrbracket\in
  \Alg\otimes\DuAlg,\]
where $\llbracket E_i,E_j\rrbracket = E_i\cdot E_j + E_j\cdot E_i$.
Specifically,
\[ \delta^1 \colon \CanonDD \to (\Alg\otimes\DuAlg) \otimes_{\IdempRing(\Alg)\otimes\IdempRing(\DuAlg)} \CanonDD\]
is given by
$\delta^1(v)=A\otimes v$

\subsubsection{Crossings.}
We characterize  the type $DA$ bimodule of a positive crossing  $\lsup{\Alg_2}\Pos_{\Alg_1}$,
where $\Alg_2=\Alg(n,\Matching_2)$, $\Alg_1=\Alg(n,\Matching_1)$,
and $\Matching_1$ is obtained from $\Matching_2$ by composing with
the transposition switching $i$ and $i+1$.  Its corresponding type
$DD$ bimodule $\lsup{\Alg_2,\DuAlg_1}\Pos$ is generated by partial
Kauffman states with the understanding that left multiplication by
$\Idemp{\x}\otimes \Idemp{\{0,\dots,2n\}\setminus \y}$ preserves the
generator corresponding to $(\state,\x,\y)$.

Decomposing the partial Kauffman states by type as explained in Example~\ref{ex:Crossing},
the differential has the following types of terms:
\begin{enumerate}[label=(P-\arabic*),ref=(P-\arabic*)]
\item 
  \label{type:OutsideLRP}
  $R_j\otimes L_j$
  and $L_j\otimes R_j$ for all $j\in \{1,\dots,2n\}\setminus \{i,i+1\}$; these connect
  generators of the same type. 
\item
  \label{type:UCP}
  $U_{j}\otimes E_{\tau(j)}$
  for all $j=1,\dots,2n$
\item
  \label{type:UCCP}
  $C_{\{\alpha,\beta\}}\otimes [E_{\tau(\alpha)},E_{\tau(\beta)}]$,
  for all  $\{\alpha,\beta\}\in\Matching_2$; these connect generators of the same type
\item 
  \label{type:InsideP}
  Terms in the diagram below that connect  generators
  of different types:
  
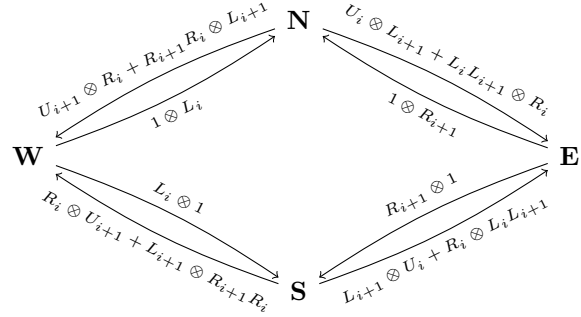
\begin{equation}
    \label{eq:PositiveCrossing}
    \begin{tikzpicture}[scale=1.8]
    \node at (0,3) (N) {$\North$} ;
    \node at (-2,2) (W) {$\West$} ;
    \node at (2,2) (E) {$\East$} ;
    \node at (0,1) (S) {$\South$} ;
    \draw[->] (S) [bend left=7] to node[below,sloped] {\tiny{$R_i\otimes U_{i+1}+L_{i+1}\otimes R_{i+1}R_i$}}  (W)  ;
    \draw[->] (W) [bend left=7] to node[above,sloped] {\tiny{$L_{i}\otimes 1$}}  (S)  ;
    \draw[->] (E)[bend right=7] to node[above,sloped] {\tiny{$R_{i+1}\otimes 1$}}  (S)  ;
    \draw[->] (S)[bend right=7] to node[below,sloped] {\tiny{$L_{i+1}\otimes U_i + R_i \otimes L_{i} L_{i+1}$}} (E) ;
    \draw[->] (W)[bend right=7] to node[below,sloped] {\tiny{$1\otimes L_i$}} (N) ;
    \draw[->] (N)[bend right=7] to node[above,sloped] {\tiny{$U_{i+1}\otimes R_i + R_{i+1} R_i \otimes L_{i+1}$}} (W) ;
    \draw[->] (E)[bend left=7] to node[below,sloped]{\tiny{$1\otimes R_{i+1}$}} (N) ;
    \draw[->] (N)[bend left=7] to node[above,sloped]{\tiny{$U_{i}\otimes L_{i+1} + L_{i} L_{i+1}\otimes R_i$}} (E) ;
  \end{tikzpicture}
\end{equation}
        For example, these give rise to terms $(1\otimes R_i)\otimes \West+(1\otimes L_{i+1})\otimes \East$ in $\delta^1(\North)$.
\end{enumerate}
The negative crossing works similarly, except that Equation~\eqref{eq:PositiveCrossing} is replaced by:
\begin{equation}
  \label{eq:NegativeCrossing}
\begin{tikzpicture}[scale=1.8]
    \node at (0,3) (N) {$\North$} ;
    \node at (-2,2) (W) {$\West$} ;
    \node at (2,2) (E) {$\East$} ;
    \node at (0,1) (S) {$\South$} ;
    \draw[->] (W) [bend left=7] to node[above,sloped] {\tiny{${U_{i+1}}\otimes{L_i}+{L_{i} L_{i+1}}\otimes{R_{i+1}}$}}  (N)  ;
    \draw[->] (N) [bend left=7] to node[below,sloped] {\tiny{${1}\otimes{R_{i}}$}}  (W)  ;
    \draw[->] (N)[bend right=7] to node[below,sloped] {\tiny{${1}\otimes{L_{i+1}}$}}  (E)  ;
    \draw[->] (E)[bend right=7] to node[above,sloped] {\tiny{${U_i}\otimes{R_{i+1}} + {R_{i+1} R_{i}}\otimes{L_i}$}} (N) ;
    \draw[->] (S)[bend right=7] to node[above,sloped] {\tiny{${R_i}\otimes{1}$}} (W) ;
    \draw[->] (W)[bend right=7] to node[below,sloped] {\tiny{${L_i}\otimes{U_{i+1}} + {R_{i+1}}\otimes{L_{i} L_{i+1}}$}} (S) ;
    \draw[->] (S)[bend left=7] to node[above,sloped]{\tiny{${L_{i+1}}\otimes{1}$}} (E) ;
    \draw[->] (E)[bend left=7] to node[below,sloped]{\tiny{${R_{i+1}}\otimes{U_{i}} + {L_i}\otimes{R_{i+1} R_{i}}$}} (S) ;
  \end{tikzpicture}
\end{equation}

\subsubsection{Local maximum}
Consider the partial knot diagram of a local maximum from Example~\ref{ex:LocalMax}.
The type $DA$ bimodule of this partial knot diagram
$\lsup{\Alg_2}\Max_{\Alg_1}$ is defined over algebras $\Alg_1$ and $\Alg_2$, and it is specified as follows.
Let $\phi_c\colon \{1,\dots,2n\}\to \{1,\dots,2n+2\}$ be the map
\begin{equation}
\label{eq:DefInsert}
\phi_c(j)=\left\{\begin{array}{ll}
j &{\text{if $j< c$}} \\
j+2 &{\text{if $j\geq c$.}}
\end{array}\right.
\end{equation}
Then,
\begin{equation}
  \label{eq:MaxAlgebras}
  \Alg_1=\Alg(n,\Matching_1)
\qquad{\text{and}}\qquad
\Alg_2=\Alg(n+1,\phi_c(\Matching_1)\cup\{c,c+1\})
\end{equation}

We specify this bimodule up to quasi-isomorphism by defining its dual
type $DD$ bimodule $\lsup{\Alg_2,\Alg_1'}\Max$.  The generators
correspond to partial Kauffman states for the partial knot diagram,
again with the convention that
$\Idemp{\x}\otimes \Idemp{\{0,\dots,2n\}\setminus \y}$ preserves the
generator corresponding to $(\state,\x,\y)$.

The differential is specified by the algebra element
$A\in \Alg_2\otimes\DuAlg_1$ 
\begin{align*}
A&=(L_{c} L_{c+1}\otimes 1) + 
(R_{c+1} R_{c}\otimes 1) 
 + \sum_{i=1}^{2n} L_{\phi(i)}\otimes R_i + R_{\phi(i)}\otimes L_i \\
& +  C_{\{c,c+1\}}\otimes 1
+ \sum_{i=1}^{2n} U_{\phi(i)}\otimes E_i 
+ \sum_{\{i,j\}\in\Matching} C_{\{\phi(i),\phi(j)\}} \otimes \llbracket E_i,E_j\rrbracket
  \end{align*}
where we have dropped the subscript $c$ from $\phi_c=\phi$.

In more detail, decomposing partial Kauffman states according to the type $\XX$, $\YY$, and $\ZZ$
specified in Example~\ref{ex:LocalMax}, the differential on the bimodule has terms  are of the following types:
\begin{enumerate}[label=($\Omega$-\arabic*),ref=($\Omega$-\arabic*)]
\item 
  \label{type:MOutsideLR}
  $R_{\phi(j)}\otimes L_j$
  and $L_{\phi(j)}\otimes R_{j}$ for all 
  $j\in \{1,\dots,2n\}\setminus \{c-1,c\}$; these connect
  generators of the same type. 
\item
  \label{type:MUC}
  $U_{\phi(i)}\otimes E_i$ for $i=1,\dots,2n$
\item 
  \label{type:MUC2}
  $C_{\{\phi(i),\phi(j)\}}\otimes \llbracket E_i,E_j\rrbracket$
  for all $\{i,j\}\in\Partition$;
\item
  $C_{\{c,c+1\}}\otimes 1$
\item 
  \label{type:MInsideCup}
  Terms in the diagram below connect  generators
  of different types.
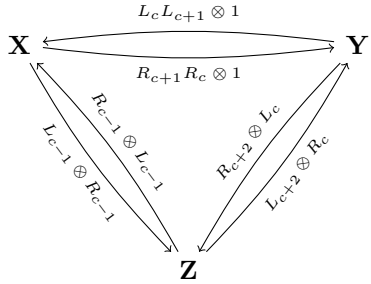
\begin{equation}
  \label{eq:CritDiag}
  \begin{tikzpicture}[scale=1.5]%[y=48pt,x=1in]
    \node at (-1.5,0) (X) {$\XX$} ;
    \node at (1.5,0) (Y) {$\YY$} ;
    \node at (0,-2) (Z) {$\ZZ$} ;
    \draw[->] (X) [bend right=7] to node[below,sloped] {\tiny{$R_{c+1} R_{c}\otimes 1$}}  (Y)  ;
    \draw[->] (Y) [bend right=7] to node[above,sloped] {\tiny{$L_{c} L_{c+1}\otimes 1$}}  (X)  ;
    \draw[->] (X) [bend right=7] to node[below,sloped] {\tiny{$L_{c-1}\otimes
R_{c-1}$}}  (Z)  ;
    \draw[->] (Z) [bend right=7] to node[above,sloped] {\tiny{$R_{c-1}\otimes L_{c-1}$}}  (X)  ;
    \draw[->] (Z) [bend right=7] to node[below,sloped] {\tiny{$L_{c+2}\otimes R_{c}$}}  (Y)  ;
    \draw[->] (Y) [bend right=7] to node[above,sloped] {\tiny{$R_{c+2}\otimes L_{c}$}}  (Z)  ;
  \end{tikzpicture}
\end{equation}
\end{enumerate}
The above description can be readily specialized to the case where the
maximum appears in an unbounded region. In these cases, there is only
one generator type, $\ZZ$.

\subsubsection{Local minimum}
Turning the above example on its top, we have $\lsup{\Alg_1}\Min^c_{\Alg_2}$, where
$\Alg_1$ and $\Alg_2$ are as in Equation~\eqref{eq:MaxAlgebras}. We specify this module
by describing its dual type $DD$ bimodule $\lsup{\Alg_1,\Alg_2'}\Min_c$. Its generators are partial Kauffman states,
with the understanding now that
$\Idemp{\x}\otimes \Idemp{\{0,\dots,2n+2\}\setminus \y}$ preserves the generator corresponding to $(\state,\x,\y)$.
The $DD$ bimodule is specified by the algebra element
$A\in \Alg_1\otimes \Alg_2'$ 
\begin{align}
A&=(1\otimes L_{c} L_{c+1}) + 
(1\otimes R_{c+1} R_{c}) 
 + \sum_{j=1}^{2n} R_{j} \otimes L_{\phi(j)} + L_{j} \otimes R_{\phi(j)} + U_{j}\otimes E_{\phi(j)} \label{eq:defDDmin}\\
&
  + 1\otimes E_{c} U_{c+1}
  + U_{\alpha}\otimes \llbracket E_{\phi(\alpha)},E_{c}\rrbracket E_{c+1} 
   + C_{\{\alpha,\beta\}}\otimes \llbracket E_{\phi(\alpha)},E_{c}\rrbracket \llbracket E_{c+1},E_{\phi(\beta)}\rrbracket. \nonumber
  \end{align}

\subsubsection{Global minimum}
When we have covered the entire diagram, save for the
last piece (the global minimum), we have a type $D$ structure $C$ over the algebra
\[ \Idemp{\{1\}}\cdot \Alg(2,1,\{1,1\})\cdot \Idemp{\{1\}}.\] After
dividing out by $C_{\{1,2\}}$, what remains can be thought of as a
chain complex over $\Field[U_1,U_2]/U_1 U_2$. Its homology  is the invariant $\KH(K)$.
Dividing out the complex by $U_2$ and taking homology gives $\KHm(K)$; and dividing out by
both $U_1$ and $U_2$ and taking homology gives $\KHa(K)$.

\subsection{Topological invariance}

It is proved in~\cite{HFKa2} that the bigraded homology modules
$\KHa(\Diag)$, $\KHm(\Diag)$ and $\KH(\Diag)$ are invariants of the
underlying oriented knot $K$ represented by the diagram $\Diag$. This
involves checking that the homology of the chain complex is invariant
under Reidemeister moves. These relations are proved locally on the level of bimodules.

For example, the bimodules of a crossing satisfy the ``braid relations''
    for any $1\leq i,j\leq 2n-1$:
  for $|i-j|>1$,
  these relations give quasi-isomorphisms of bimodules
  \begin{equation}\Pos^j\DT\Pos^i
    \simeq~
    \Pos^i\DT\Pos^j;
    \label{eq:FarBraids}
    \end{equation}
    (where we have suppressed the algebras which come naturally from the pictures)
    while if $j=i+1$,
    then,    
    \begin{equation}
    \Pos^i\DT\Pos^{i+1}\DT\Pos^i
      \simeq~\Pos^{i+1}\DT\Pos^{i}
      \DT\Pos^{i+1}.
    \label{eq:NearBraids}
    \end{equation}
    (again for suitably chosen algebras).
Thus, we can think of these bimodules as giving a braid group action on the derived
category of modules over $\Alg(n,\Matching)$; compare~\cite{KhovanovSeidel, Bimodules, Manion}.

The knot invariants $\KHa(K)$, $\KHm(K)$, and $\KHwz(K)$ are designed
to agree with their knot Floer homological analogues. One can
nonetheless, study them independently of holomorphic methods.  For
example, one can verify certain fundamental properties
within the algebraic realm: 
relating their graded Euler characteristics with the 
Alexander polynomial of $K$, establishing  a K{\"u}nneth formula for connected
sums, and verifying an algebraic structure result for $\KHm(K)$ analogous to
Proposition~\ref{prop:Structure}; see~\cite{HFKa2}.

\section{Bordered knot algebras and pseudo-holomorphic curves}

In fact,  we prove that this bordered invariant is
equivalent to knot Floer homology~\cite{HolKnot}.
To establish the link between the algebraic constructions and knot
Floer homology, it is useful to give a pseudo-holomorphic
interpretation of these structures.

Upper knot diagrams can be represented by suitably decorated (partial) Heegaard diagrams.
An {\em upper Heegaard diagram} is a surface $\Sigma$ of genus $g$ and
$2n$ boundary components, labelled $Z_1,\dots,Z_{2n}$, together with the following additional data:
\begin{itemize}
\item  A collection
of disjoint, embedded arcs $\{\alpha_i\}_{i=1}^{2n-1}$, so that
$\alpha_i$ connects $Z_i$ to $Z_{i+1}$.
\item A collection of 
disjoint embedded closed curves $\{\alpha^c_i\}_{i=1}^{g}$
(which are also disjoint from $\alpha_1,\dots,\alpha_{2n-1}$).
\item A collection of embedded, mutually disjoint closed curves
$\{\beta_i\}_{i=1}^{g+n-1}$.
\end{itemize}  Both sets of $\alpha$-and the
$\beta$-circles are required to consist of homologically linearly
independent curves, and the $\beta$-circles are further
required to have the following combinatorial property: the surface
obtained by cutting $\Sigma$ along $\beta_1,\dots,\beta_{g+n-1}$,
which has $n$ connected components, is required to contain exactly
two boundary circles in each component. This requirement gives a
matching $\Matching$ on $\{1,\dots,2n\}$ (a partition into two-element subsets), where $\{i,j\}\in \Matching$ if $Z_i$ and $Z_j$
can be connected by a path that does not cross any $\beta_k$.

We sometimes abbreviate the data
\[\Hup=(\Sigma,Z_1,\dots,Z_{2n},\{\alpha_1,\dots,\alpha_{2n-1}\},\{\alpha^c_1,\dots,\alpha^c_{g}\},                                                                               
\{\beta_1,\dots,\beta_{g+n-1}\}),\]
and let $\Matching(\Hup)$ be the induced matching.

\begin{figure}[h] 
                  \centering \input{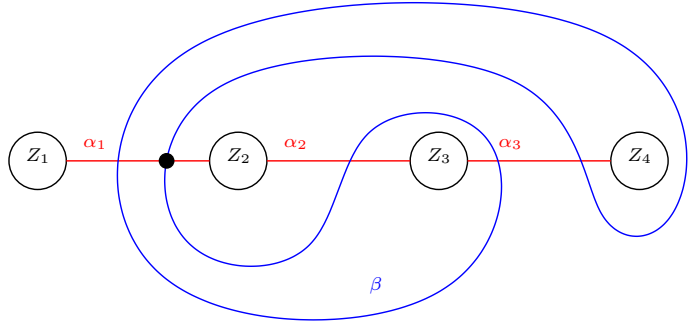} \caption{{\bf
  Upper Heegaard diagram.} The diagram here is the upper Heegaard
  diagram for the upper planar diagram from
  Figure~\ref{fig:UpperDiagTref}; the black dot represents a Heegaard
  state corresponding to the Kauffman state from
  Figure~\ref{fig:UpperDiagTref}.} \label{fig:UpperHeeg} \end{figure}

An {\em upper Heegaard state} is an $g+n-1$-tuple of points $\x$,
each of which is in $\alpha_i\cap \beta_j$ or $\alpha^c_i\cap \beta_j$ for various $i$ and $j$, 
so that each circle $\alpha_i$ contains an element in $\x$, each $\beta_j$ contains 
an element in $\x$, and no more than one element in $\x$ is contained on each 
$\alpha$-arc $\alpha_i^c$.  Each Heegaard state $\x$
determines a subset $s(\x)$ of $\{1,\dots,2n\}$ with cardinality $n$, or, equivalently,
an idempotent $\Idemp{s(\x)}$ in $\Alg(n,\Matching)$: 
\[ s(\x)=\{1,\dots,2n-1\}\setminus\{1\leq i\leq 2n-1 \big| \x\cap \alpha_i~\text{is non-empty}\}.\]

 \begin{figure}[h] \centering \input{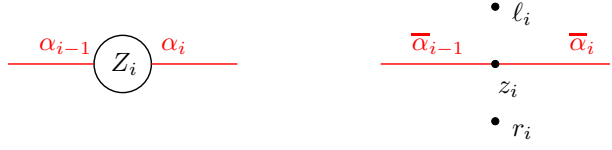}
 \caption{{\bf Boundary markings.}  On the left, we have shown a neighborhood of a boundary component
$Z_i$ of $\Sigma$. To the right, we have filled in $Z_i$, replacing it with the point $z_i$.}
 \label{fig:BoundaryMarkings}
 \end{figure}

Fill in each boundary component  $Z_i$, to obtain a closed Riemann surface ${\overline\Sigma}$,
with $2n$ marked points $z_i$.
Extend $\alpha_i$ into ${\overline\Sigma}$ to obtain a curve ${\overline\alpha}_i$
with $\partial {\overline \alpha}_i=z_i-z_{i-1}$, as shown in Figure~\ref{fig:BoundaryMarkings}.
We will place a pair of points $r_i$ and $\ell_i$ in a neighborhood of $z_i$, separated
by ${\overline\alpha}_{i-1}\cup{\overline\alpha}_i$. (In the special case where $i=1$ and $2n$, the two points $r_i$ and $\ell_i$
are not separated by this neighborhood, as one of $\alpha_{i-1}$ or
$\alpha_i$ does not exist.)  We will be working with holomorphic disks
in $\Sym^{g+n-1}(\Sigma)$, relative to $\Tb$ and
\[ L^0_{\alpha}=\alpha_1^c\times\dots\times\alpha^c_g\times \Sym^{n-1}({\overline\alpha}_1\cup\dots\cup{\overline\alpha}_{2n-1}).\]
Note that $L^0_\alpha$ is a singular space, with singularities
contained in the locus where two points are contained on the same
$\alpha_i$-curve. We will work away from this locus, in the subspace
$L_\alpha\subset L^0_{\alpha}$ consisting of those $n-1$-tuples where
no two points lie on the same $\alpha_i$ (this corresponds to the
``boundary monotonicity'' condition of~\cite{InvPair}), and each point
lies in the interior of some $\alpha_i$, denoted
$\Interior{\alpha_i}$. Clearly, $L_\alpha$ is disconnected; in fact
\[ L_{\alpha}=
\alpha^c_1\times \dots\times \alpha^c_{g}
\times \left(\bigcup_{\{t_1,\dots,t_{n-1}\}\subset \{1,\dots,2n-1\}} \Interior{\alpha_{t_1}}\times\dots\times\Interior{\alpha_{t_{n-1}}}\right).\]
The upper Heegaard states correspond to intersection points of $L_\alpha$ with $\Tb$.
If $\x$ is in the component of $L_\alpha$ specified by $\{t_1,\dots,t_{n-1}\}$, then 
$\Idemp{s(\x)}\cdot \x = \x$ where $s(\x)=\{1,\dots,2n-1\}\setminus\{t_1,\dots,t_{n-1}\}$.

Let $\pi_2(\x,\y)$ denote the space of homotopy classes of Whitney
disk as before, only now half of the boundary of the disk is mapped
into $L^0_\alpha$, and the other half into $\Tb$.  (In fact, the disks
of interest to us will have half their boundary mapped into the
closure of $L_\alpha$, ${\overline L}_{\alpha}\subset L^0_\alpha$.)
Each $\phi\in\pi_2(\x,\y)$ with non-negative local multiplicities
determines an algebra element
$b_0(\phi)\in\Idemp{\x}\cdot \Blg_0\cdot\Idemp{\y}$, given by
\[ b_0(\phi)=U_1^{c_1(\phi)}\cdots U_{2n}^{c_{2n}(\phi)}\cdot
\gamma_{\x\y},\] where
$c_i(\phi)=\min(n_{\ell_i}(\phi),n_{r_i}(\phi))$.  

Let $X$ denote the vector space spanned by upper Heegaard states.
Consider the map
\[ \mfacZ\colon X \to \Blg_0\otimes X \]
(again, where the tensor product is taken over the idempotent ring)
defined by 
\[ \mfacZ(\x)=\sum_{\y\in \States(\HD)} \sum_{\{\phi\in\pi_2(\x,\y)\big|\Mas(\phi)=1\}}
 \#\left(\frac{\ModFlow(\phi)}{\R}\right) \cdot b_0(\phi)\cdot \y.\]

\begin{prop}
  The endomorphism $\mfacZ$ satisfies the structure relation
  \[ (\mu_2^{\Blg_0}\otimes \Id_{X})\circ (\Id_{\Blg_0}\otimes \mfacZ)\circ \mfacZ(\x) -
  \left(\sum_{\{i,j\}\in\Matching} U_i U_j\right)\otimes \x + \Ideal \otimes X,\]
  where $\Ideal$ is the ideal used in the definition of $\Blg(2n,n)$.
\end{prop}

\begin{proof}[Sketch of proof]
In broad terms, the proof of this is the usual $\partial^2=0$ proof in Lagrangian Floer homology:
it is proved by considering one-dimensional moduli spaces of pseudo-holomorphic disks, and identifying
their boundaries.

In more detail, the proof rests on the following observations

{\bf Observation 1.} First, note that the map $b_0$ is additive under juxtapositions, in
the sense that if $\x,\y,\z\in L_\alpha\cap\Tb$,
$\phi_1\in\pi_2(\x,\y)$, and $\phi_2\in\pi_2(\y,\z)$ are two homotopy
classes whose local multiplicities at all the $\ell_i$ and $r_i$ are
non-negative, then 
\begin{equation}
        \label{eq:AdditiveB0}
        b_0(\phi_1*\phi_2)=b_0(\phi_1)\cdot b_0(\phi_2).
\end{equation}
This follows quickly from the fact that for any
$\phi\in\pi_2(\x,\y)$, 
\[
n_{\ell_i}(\phi)-n_{r_i}(\phi)=v^{s(\x)}_i-v^{s(\y)}_i,\]
together with the additivity of local multiplicities under juxtapositions; i.e.
  \[n_{p}(\phi_1*\phi_2)=n_{p}(\phi_1)+n_{p}(\phi_2)\]
  for any $p\in \Sigma$,
  and the definition of multiplication in the algebra.

{\bf Observation 2.}  The next point is that if $\phi=\phi_1*\phi_2$,
where $\phi_1\in\pi_2(\x,\y)$ and $\phi_2\in\pi_2(\y,\z)$ for some
$\y\in L_\alpha\cap\Tb$ has an alternative decomposition
$\phi=\phi_1'*\phi_2'$ with $\phi_1\in\pi_2(\x,\y')$ and
$\phi_2\in\pi_2(\y',\z)$ with $\y\in (L_\alpha^0\cap\Tb)\setminus
(L_\alpha\cap \Tb)$, then $b_0(\phi)\in \Ideal$. To see why, we refer
to Figure~\ref{fig:TypeDrels}.   At
the left, the pair $\{x_1,x_2\}$ represents an upper state $\x$,
$\{x_1,y_2\}$ represents part of an upper state $\y$, and
$\{y_1,y_2\}$ represents part of an upper state $\z$. The small bigon
near $Z_{i+1}$ gives a term of $L_{i+1}\otimes \y$ in $\mfacZ(\x)$;
and the small bigon near $Z_i$ gives a term of $L_i\otimes \z$ in
$\mfacZ(\y)$. Since $L_{i+1} L_{i}=0$, we do not need to consider the
ends of the moduli space from $\x$ to $\z$: the corresponding term is
in the ideal $\Ideal$. Note that the alternative factorization of this moduli
spaces involves $\{y_1,x_2\}$ which is not in a
$L_\alpha\cap\betas$. At the right is a similar picture, now with $\x$
containing $\{x,t\}\subset \x$, $\{y,t\}\subset \y$, and
$\{z,t\}\subset \z$.  A small bigon from $x$ to $y$ gives a term of
$\y$ in $\mfacZ(\x)$. The bigon from $y$ to $z$ containing $Z_i$ gives
a term of $U_i\otimes \z$ in $\mfacZ(\y)$, which is in the ideal $\Ideal$.

\begin{figure} \centering \input{TypeDrels.pstex_t}
  \caption{{\bf Relations in $\Blg$.}}
\label{fig:TypeDrels}
\end{figure}

In view of the above two observations, it suffices to consider ends of
moduli spaces $\phi\in\pi_2(\x,\z)$ for which all broken flowline
decompositions $\phi=\phi_1*\phi_2$ with $\phi_1\in\pi_2(\x,\y)$ and
$\phi_2\in\pi_2(\x,\y)$ have $\y\in L_\alpha\cap\Tb$.  The usual
$\partial^2=0$ proof now shows that the number of such ends has the
same parity as the number of boundary degenerations: holomorphic
curves which have boundary contained entirely on $L^0_\alpha$ or $\Tb$.
We complete the proof with two more observations:
  
{\bf Observation 3.}
A homotopy class corresponding to curves with boundary contained entirely in ${\overline L}_\alpha$
has positive coefficients at every $\ell_i$ and $m_i$; thus, the associated algebra element lies in the
ideal $\Ideal$.

And finally, to keep track of the $\beta$-boundary degenerations, we have:

{\bf Observation 4.}
There are $n$ homotopy classes of disks $\psi$ with boundary in $\Tb$,
corresponding to the matchings $\{i,j\}\in\Matching$, and their
corresponding algebra element is $U_i\cdot U_j$; see Figure~\ref{fig:BoundaryDegeneration}.
\begin{figure} \centering \input{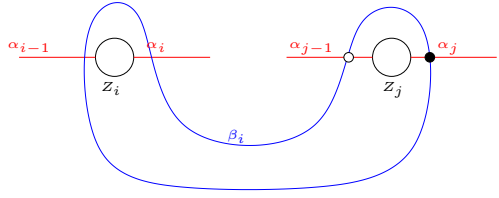}
  \caption{{\bf Motivation for introduction $C_{i,j}$.}
  }
\label{fig:BoundaryDegeneration}
\end{figure}
\end{proof}

The map $\mfacZ$ induces a map
\[ \mfac\colon X \to \Blg\otimes X\]
satisfying the structure relation
\[ (\mu_2\otimes\Id_X)\circ(\Id_{\Blg}\otimes\mfac)\circ\mfac =
\sum_{\{i,j\}} U_i U_j,\] where here $\mu_2$ is multiplication in
$\Blg$.  Note that since $\mu_1=0$ in $\Blg$, the structure relation
for $\mfac$ is nearly the type $D$ structure relation: it would be if
the right hand side were zero. Instead, this structure can be thought
of as a ``curved'' type $D$ structure (i.e. for an algebra with a
$\mu_0$ operation); compare~\cite{KhovanovRozansky}.

We can turn such an object into a type $D$ structure over $\Alg$, defining:
\[ \delta^1(\x)= \left(\sum_{\{i,j\}\in\Matching} C_{\{i,j\}}\right)\otimes \x + \mfac(\x).\]

This defines the type $D$ structure of an upper diagram. 

We consider the upper diagram from Figure~\ref{fig:UpperHeeg}. Note
that we are now working in the (first symmetric product of) the
two-sphere. In some sense, the type $D$ structure is capturing the
Lagrangian Floer homology of an interval with four marked points on it
and the closed curve $\beta$. The type $D$ structure has five
generators corresponding to the five intersection points, which we
label from left to right in the picture $x_1,x_2,t,y_1,y_2$.
\begin{equation}
  \label{eq:Dmod}
  \begin{tikzpicture}[scale=1.8]
  \node at (0,0) (x1) {$x_1$}   ;
  \node at (1,0) (x2) {$x_2$} ;
  \node at (2,0) (t) {$t$} ;
  \node at (3,0) (y1) {$y_1$} ;
  \node at (4,0) (y2){$y_2$} ;
    \draw[->] (x1) [bend left=15] to node[above,sloped] {\tiny{$U_4$}}  (x2)  ;
    \draw[->] (x2) [bend left=15] to node[below,sloped] {\tiny{$U_3$}}  (x1)  ;
    \draw[->] (x2) [bend left=15] to node[above,sloped] {\tiny{$L_2 U_1$}}  (t)  ;
    \draw[->] (t) [bend left=15] to node[below,sloped] {\tiny{$R_2$}}  (x2)  ;
    \draw[->] (t) [bend left=15] to node[above,sloped] {\tiny{$L_3$}}  (y1)  ;
    \draw[->] (y1) [bend left=15] to node[below,sloped] {\tiny{$R_3 U_4$}}  (t)  ;
    \draw[->] (y1) [bend left=15] to node[above,sloped] {\tiny{$U_2$}}  (y2)  ;
    \draw[->] (y2) [bend left=15] to node[below,sloped] {\tiny{$U_1$}}  (y1)  ;
    \draw[->] (x2) [bend left=40] to node[above,sloped] {\tiny{$L_2 L_3$}} (y2) ;
    \draw[->] (y1) [bend left=40] to node[below,sloped] {\tiny{$R_3 R_2$}} (x1) ;
  \end{tikzpicture}
\end{equation}

With a little more work, one can define the $\Ainfty$ module
associated to a lower diagram.  In this case, the higher actions count
pseudo-holomorphic disks that go out to the $\alpha$-boundary.  The
algebra actions record the sequence of walls in which a
pseudo-holomorphic disk crosses the walls in $L_\alpha$. Like in
bordered Floer homology, it is clearer to express these actions
in the language of Lipshitz's cylindrical reformulation of
Heegaard Floer homology~\cite{LipshitzCyl}. A pairing theorem for
recapturing knot Floer homology is then proved using a ``time dilation'' argument analogous to the bordered case
(see~\cite[Chapter~9]{InvPair}), with a little extra attention paid now to
$\beta$-boundary degenerations.

Working out the the type $DD$ bimodules for basic pieces is a fairly
straightforward matter. Extending the pairing theorem to type $DA$ bimodules then gives the following:
\begin{thm}~\cite{HolKnot}
  If $K\subset S^3$ is a knot, then there are isomorphisms of bigraded modules
  \[ \HFKa(K)\cong \KHa(K)\qquad \HFKm(K)\cong \KHm(K)\qquad \HFKwz(K)\cong\KH(K).\]
\end{thm}

\section{Further remarks}

To effectively compute the chain complexes $\KC(K)$, one can start
with the type $D$ structure corresponding to the global maximum, and
successively enlarge it, moving down the knot projection. The computation is
significantly improved by eliminating
(by passing to a homotopy equivalent complex) generators $\x$ with
\[ \delta^1(\x) = \sum_{\y} a_{\x,\y}\otimes \y \] for which some
$a_{\x,\y}$ is an idempotent. Another simplification is achieved by
working directly with the operators $\mfac$; see~\cite{HFKa2}.

Throughout the above discussion, we worked in characteristic $2$ to
avoid signs.  In fact, bordered knot Floer homology with $\Z$
coefficients can be worked out, after paying a little extra care to
sign conventions. This is done in~\cite{HFKa2}. One motivation
is to find a knot $K$ in $S^3$ whose knot Floer homology (with
$\Z$ coefficients) has torsion. Despite a rather extensive search, we
have not yet found such a knot.

Long before the discovery of knot Floer homology, Andreas Floer
proposed a construction of a knot invariant using
instantons~\cite{FloerKnots}. In~\cite{KMknotsSutures}, Kronheimer and
Mrowka conjectured an isomorphism between Floer's construction and
knot Floer homology $\HFKa$ (taken now with $\Q$ coefficients).
This would provide some link between the fundamental group and knot Floer homology.

The knot Floer homology package described here are sufficient for many computations:
calculating $\tau(K)$, $\epsilon(K)$, and $\HFa$ of surgeries on $K$. 
To understand the function $\Upsilon(K)$ and $\HFm$ of surgeries on $K$,
one needs to understand the knot Floer complex with more structure (in
effect, without the $UV=0$ specialization from
Equation~\eqref{eq:CFKwz}). To study this invariant, one must work with
a larger algebra; see~\cite{Pong}.

\bibliographystyle{plain}
\bibliography{biblio}

\end{document}